\documentclass[11pt]{amsart}
 \pdfoutput=1 
\usepackage{marginnote}
\usepackage{geometry}
\usepackage{cancel}
\usepackage{amsmath,amsthm, yhmath}
\allowdisplaybreaks[1]
\usepackage{amssymb,mathrsfs}

\usepackage{tensor}

\usepackage{slashed}
\usepackage{placeins} 

\usepackage{wrapfig}
\usepackage[pdftex]{graphicx}
\usepackage{pdfpages}

\usepackage{lmodern}
\usepackage[T1]{fontenc}

\usepackage{esint}
\usepackage{epstopdf}
\usepackage{amscd}
\usepackage{enumerate}
\usepackage{mymathsty}
\usepackage{fancyhdr} 
\usepackage{microtype}
\usepackage{hyperref}
\usepackage[all]{xy}
\xyoption{poly}

 \numberwithin{equation}{section}
\usepackage{array}
\newcolumntype{L}[1]{>{\raggedright\let\newline\\\arraybackslash\hspace{0pt}}m{#1}}
\newcolumntype{C}[1]{>{\centering\let\newline\\\arraybackslash\hspace{0pt}}m{#1}}
\newcolumntype{R}[1]{>{\raggedleft\let\newline\\\arraybackslash\hspace{0pt}}m{#1}}

	\DeclareMathOperator{\grad}{grad}

	\DeclareMathOperator{\Hom}{Hom}

	\newcommand*{\op}[1]{{\mathrm{#1}}}
	
    \newcommand{\modop}{\slashed \bigtriangleup} 
    \newcommand{\logmodop}{\slashed \bigtriangledown}

	\DeclareMathSymbol{\Gamma}{\mathalpha}{operators}{0}
	\DeclareMathSymbol{\Delta}{\mathalpha}{operators}{1}
	\DeclareMathSymbol{\Theta}{\mathalpha}{operators}{2}
	\DeclareMathSymbol{\Lambda}{\mathalpha}{operators}{3}
	\DeclareMathSymbol{\Xi}{\mathalpha}{operators}{4}
	\DeclareMathSymbol{\Pi}{\mathalpha}{operators}{5}
	\DeclareMathSymbol{\Sigma}{\mathalpha}{operators}{6}
	\DeclareMathSymbol{\Upsilon}{\mathalpha}{operators}{7}
	\DeclareMathSymbol{\Phi}{\mathalpha}{operators}{8}
	\DeclareMathSymbol{\Psi}{\mathalpha}{operators}{9}
	\DeclareMathSymbol{\Omega}{\mathalpha}{operators}{10}

	\renewcommand{\epsilon}{\varepsilon}
	
	\DeclareMathSymbol{\T}{\mathbin}{AMSb}{"54}

\usepackage{mathtools}
\newcommand{\defeq}{\vcentcolon=}

\newcommand{\symd}{\partial}

\def\Xint#1{\mathchoice
      {\XXint\displaystyle\textstyle{#1}}%
      {\XXint\textstyle\scriptstyle{#1}}%
      {\XXint\scriptstyle\scriptscriptstyle{#1}}%
      {\XXint\scriptscriptstyle\scriptscriptstyle{#1}}%
      \!\int}
      \def\XXint#1#2#3{{\setbox0=\hbox{$#1{#2#3}{\int}$}
        \vcenter{\hbox{$#2#3$}}\kern-.5\wd0}}
   
   \def\dashint{\Xint-}

\begin{document}

 \author{Yang Liu}
\address{Max Planck Institute for Mathematics,
Vivatsgasse 7,
53111, Bonn,
Germany}
\email{liu.858@osu.edu}

\title{Scalar curvature in conformal geometry of Connes-Landi noncommutative manifolds}

\keywords{Connes-Landi deformation, toric noncommutative manifolds, pseudo differential calculus, modular curvature, heat kernel expansion}
\date{\today}

\begin{abstract}
    

We first propose a conformal geometry for Connes-Landi noncommutative manifolds and study the associated scalar curvature. The new scalar curvature  contains its Riemannian counterpart as the commutative limit. Similar to the results on noncommutative two tori, the quantum part of the curvature consists of actions  of the modular derivation through two local curvature functions. Explicit expressions for those functions are obtained for all even dimensions (greater than two). In dimension four, the one variable function shows striking similarity to
 the analytic functions of the characteristic classes appeared in the Atiyah-Singer local index formula, namely, it is roughly a product of the $j$-function (which defines the $\hat A$-class of a manifold) and an exponential function (which defines the Chern character of a bundle). By performing two different computations for the variation of the Einstein-Hilbert action, we obtain a deep internal relations between two local curvature functions. Straightforward verification for those relations gives a strong conceptual confirmation for the whole computational machinery we have developed so far, especially the Mathematica code hidden behind the paper.
 

%

\end{abstract}

\maketitle

\tableofcontents


\section{Introduction}

The general question behind this paper is to explore the notion of  intrinsic curvature, which lies in the core of the geometry, in the operator theoretical framework (noncommutative differential geometry). The question remained intangible until the recent development of modular geometry on noncommutative two tori  \cite{MR3194491}, other major references on this subject include  \cite{leschdivideddifference,MR2907006,MR2956317,MR3148618,MR3540454}. The computation has been
extended
to noncommutative four tori smoothly \cite{MR3369894,MR3359018}. The essential computational tool is Connes's pseudo differential calculus to $C^*$-dynamical system, first constructed in \cite{connes1980c} and upgraded in \cite{MR3540454} to Heisenberg modules. Some  different approaches can be found in \cite{2013SIGMA...9..071R,Bhuyain2012}. 
\par

In the modular geometry on noncommutative tori, the Riemannian aspect is somehow hidden in the sense that the original metric is flat.  To obtain a  stronger demonstration that our approach does include Riemannian geometry
as a special case, we would like to test the ideas on a larger class of deformed Riemannian manifolds known as Connes-Landi noncommutative manifolds, first introduced  in \cite{MR1846904}. The computation was initiated in the author's previous work \cite{Liu:2015aa}. The first main result in  \cite{Liu:2015aa} is a pseudo differential calculus which is suitable for studying the spectral geometry. Such a calculus not only records the Riemannian curvature information but also dramatically simplify the computation. By
testing the calculus on the  scalar Laplacian operator $ \Delta_\varphi$ in \cite{MR3194491},  we recovered the local curvature functions. This paper starts with a unfinished problem in \cite{Liu:2015aa}, namely, computing  the full scalar curvature for all even dimensional Connes-Landi noncommutative manifolds. \par

Let us first recall some basic notions in conformal geometry in the operator theoretical setting. The generalization to the noncommutative setting is straightforward, which leads to the conformal geometry of Connes-Landi noncommutative manifolds defined in Table \ref{tab:defn-curvature}.
\par
For a closed Riemannian manifold $(M,g)$, which is also spin with the spinor bundle $\slashed S_g$ and the spinor Dirac operator $\slashed D$, the spin geometry can be recovered by the Dirac model, consists of the spectral data  $(C^\infty(M), L^2(\slashed S_g), \slashed D)$. This motivates the basic paradigm of noncommutative geometry: the notion of spectral triples $(\mathcal A, \mathcal H, D)$, in which $\mathcal A$ is the coordinate algebra and the operator $D$ (playing the role of
the Dirac operator) encodes the metric (cf. \cite{MR1303779,MR1334867}). Thanks to the conformal covariant property of the spinor Dirac operator $\slashed D$, the conformal change of metric $g' = e^{-2h} g$, where $h \in C^\infty(M)$ real-valued, in the spectral setting can be achieved by replacing $\slashed D$ by $D_h =e^{h/2}\slashed D e^{h/2}$. The perturbed Dirac model $(C^\infty(M), L^2(\slashed S_g), D_h)$ via a conformal factor $e^h$ is a natural example of twisted spectral triples of type III in \cite{MR2427588}.
\par

Once the metric is fixed (as $\slashed D$ or $D_h$),  local invariants, such as the scalar curvature function, are encoded in the heat kernel asymptotic: 
\begin{align}
    \Tr(f e^{-t\slashed D^2}) \backsim_{t \searrow 0} \sum_{j=0}^\infty V_j(f,\slashed D^2) t^{(j-m)/2}, \,\, \, \forall f\in C^\infty(M),
    \,\,\, m=\dim M.
    \label{eq:introv2-heatasym}
\end{align}
Each coefficient $V_j(f,\slashed D^2)$ is a spectral  functional (in $f$). The associated local expressions $V_j(x,\slashed D^2) \in \Gamma(\End(\slashed S))$ (also called functional densities) are defined by the property: 
 \begin{align*}
    V_j(f,\slashed D^2) = \varphi_0(f V_j(x,\slashed D^2)) , \,\,\, \forall f\in C^\infty(M),
\end{align*}
here $\varphi_0(\psi) = \int_M \Tr_x(\psi) \mathrm d\op{vol}_g$, $\forall \psi \in \Gamma(\End(\slashed S))$, where $\Tr_x$ is the fiberwise trace. Upto an overall constant $(4\pi)^{-m/2}$, $V_0(x,\slashed D^2) = I$ which is related to the volume form and the next term recovers the scalar curvature function $\mathcal S_{\slashed D}$:
\begin{align}
    V_2(x,\slashed D^2) =- \frac{1}{12} \mathcal S_{\slashed D}. 
    \label{eq:introv2-V2term-local}
\end{align}
See \cite{connes2008noncommutative}, Section 11.1 for a detailed discussion. 
\begin{table}[!htbp]
\center
\begin{tabular}{|C{5cm}|C{6cm}|}\hline
    Riemannian Geometry & Connes-Landi spectral triples  \\ \hline
 $e^{-2h}$ with $h \in C^\infty(M)$ real-valued  &
   $e^{-2h}$ with $h=h^* \in C^\infty(M_\theta)$ self-adjoint\\ \hline
  Original metric $g$ & Spinor Dirac $\slashed D$ \\ \hline
  $g' = e^{-2h} g$ &  $D_h = e^{h/2}\slashed D e^{h/2}$ \\ \hline
 Scalar curvature for $g'$: $\mathcal{S}_{g'}$ 
 & $R_{D_h} \in C^\infty(M_\theta)$: local expression of $V_2(\cdot, D_h^2)$ defined in \eqref{eq:introv2-V2term-local} \\
 \hline
\end{tabular}
\caption{Conformal change of metric and the associated scalar curvature in the noncommutative setting.}
\label{tab:defn-curvature}
\end{table}
%
%

\par

We are ready to state the main progresses made in this paper. Assume that the dimension $m$ of the manifolds is alway even a greater than two, the requirement comes only from Lemma \ref{lem:modcur-put-lamda-to-1}.  Similar to the results on noncommutative two tori, the scalar curvature in the conformal geometry of Connes-Landi noncommutative manifolds is of the form:
\begin{align}
            \label{eq:introv2-RDlogk}
        R_{ D_h} &= \, e^{(-m+2)h} \brac{
            \mathcal K^{(m)}_{ D_h}(\logmodop) (\nabla^2 h) + 
    \mathcal H^{(m)}_{ D_h}(\logmodop_{(1)}, \logmodop_{(2)})(\nabla h \nabla h)
} g^{-1}  \\
&+c_{\Delta} e^{(-m+2)h} \mathcal S_\Delta , \,\,\,m=\dim M. \nonumber
    \end{align}
     Upto a volume factor $e^{-mh}$, \eqref{eq:introv2-RDlogk} contains its Riemannian version \eqref{eq:scalcur-chofscalcur-commutative} as the commutative limit.
    The noncommutative feature  is reflected by the action of the modular derivation $\logmodop = -2\op{ad}_h$ through two local curvature functions $\mathcal K^{(m)}_{ D_h}(u)$ and $\mathcal H^{(m)}_{ D_h}(u,v)$, which are still begging for more conceptual understanding.
    \par 
    The dependence of the local curvature functions $\mathcal K_{ D_h}(u)$ and $\mathcal H_{D_h}(u,v)$ on the dimension $m$ is due to integration over the unit sphere $\mathbb S^{m-1}$ combined with some integration by parts arguments. It turns out that they are all contained in the germs (at $u=0$) of function $F$ and $G$ defined in  \eqref{eq:scalcur-Ktotal} and \eqref{eq:scalcur-Htotal} respectively.  In dimension four, the functions can be derived from
    $F$ and $G$ at $u=0$; in  dimension six, one needs to compute the first jet of $F$ and $G$ at $u=0$;  in dimension eight, one needs the second jet, etc.
    
    \par

For the Gaussian curvature of noncommutative two tori (\cite{MR3194491}, Theorem 4.8), there are two celebrated features for the local curvature functions: 1) The one variable function is the generating function of Bernoulli numbers. 2) There is an  intriguing functional relation  (Eq. 4.38 in \cite{MR3194491}) between them. The two facts both have generalizations in higher dimensions.     

\par    
In dimension four, the one-variable local curvature function equals:
   \begin{align}
       \label{eq:introv2-K}
        \mathcal K_{ D_h}(u) = -\frac12 e^{u/2} \frac{  \sinh \left(u/4\right)}{(u/4)}. 
    \end{align}
One can see the striking similarity to  the analytic functions for the characteristic classes appeared in the Atiyah-Singer local index formula. The $j$-function $\sinh(x/2)/(x/2)$ defines the $\hat A$-class of a manifold while the exponential function gives rise to the Chern character of a vector bundle. It is not a coincident. Indeed, the $j$-function, or $f_1(u) =(e^u-1)/u$, the reciprocal of the generating functions of Bernoulli numbers, they play a crucial role in a noncommutative version of ``chain rule'' in different areas of mathematics. In local index formula, the $\hat{A}$-genus naturally through  the Jacobian of the exponential map of a Lie group $G$: in exponential coordinates $d\exp a = j_{\mathfrak{g}}(h)dh$, where $\mathfrak{g}$ is the Lie algebra, $dh$ is a Haar measure  and
\begin{align*}
j_{\mathfrak{g}} =  \mathrm{det}_{\mathfrak{g}} \brac{ -
f_1(\op{ad}_h )}=\mathrm{det}_{\mathfrak{g}} \brac{ -
\frac{e^{\op{ad}_h}-1}{\op{ad}_h}
  }. 
\end{align*}
See Proposition 5.1 in \cite{berline1992heat}. In our operator theoretical calculus, the $j$-function arises from  differentiating  the conformal factor $e^h$:
\begin{align}
    \nabla e^h  = f_1(\op{ad}_h)(\nabla h) = e^h \frac{e^{\op{ad}_h}-1}{\op{ad}_h}(\nabla h).
    \label{eq:introv2-f1}
\end{align}
 For the other factor $e^{u/2}$ in \eqref{eq:introv2-K}, our understanding is only at the computational level. The factor depends on the dimension $m =\dim M$, it is a sum of exponential  functions in higher dimensions. The dependence on the dimension  reflects the fact that if we rescale the metric by a factor $e^{-2h}$, then the volume is rescaled by  $e^{-mh}$. On the other hand, it also somehow counts how many times do we have to move $e^h$ in front of $\nabla e^h$ during the
 lengthy calculation. It is worth mentioning that in the theory of characteristic classes, those functions are used as  formal power series. It is quite surprising that in our context, they are treated as actual functions. For instance, in \cite{MR3194491}, the positivist of the function $\tilde K_0$ plays a key role.

\par

With the local expression \eqref{eq:introv2-RDlogk}  in hand, the next thing to study is the Einstein-Hilbert action $F_{\mathrm{EH}}(h) \defeq V_2(1,D_h^2)= \varphi_0(R_{ D_h})$. The related variation problem lead to a functional relation between  $\mathcal K^{(m)}_{ D_h}$ and $\mathcal H^{(m)}_{ D_h}$ in \eqref{eq:introv2-RDlogk}, which is the second main result of this paper. The internal relations gives strong conceptual verification for the whole machinery we have developed to compute the curvature, especially for the \textsf{Mathematica} code which is hidden behind the paper. \par
 
 What is behind the internal relations is the variation equation \eqref{eq:scalcur-varofheatop} which implies that one can fully recover the scalar curvature functional $f \mapsto V_2(f,D_h^2)$, where $f \in C^\infty(M_\theta)$ by the Einstein-Hilbert functional $h \mapsto V_2(1,D_h^2)$. The former is a functional on coordinate functions while the later is a functional on metrics. Take $f =1$ to be the constant function, one recovers the Einstein-Hilbert functional and the functions $\mathcal K^{(m)}_{ D_h}(u)$ and $\mathcal H^{(m)}_{ D_h} (u,v)$ are compressed into a function $T(u)$ defined in Prop. \ref{prop:scalcur-EHlocal}. To unpack the fully scalar curvature functional from the Einstein-Hilbert functional (or from $T(u)$), one has to compute the variation, which consists of some basic operation in operator version of calculus (in the sense of \cite{leschdivideddifference}),  such as integration by part, applying the chain rule, swapping  different types of derivation: $\nabla$ and $\logmodop$. They all reflected in the local curvature functions $L$ to $Q$ defined in \eqref{eq:L-scalcur} to  \eqref{eq:Q-scalcur}. The author believes that  those are the conceptual reasons for the functional relations asked in Connes-Moscovici \cite{MR3194491}. Another remark is that in the case of  dimension two, the variation equation \eqref{eq:scalcur-varofheatop} since 
  the factor $(\dim M -2)$ in the right hand side vanishes. The factor can be normalized by passing from the heat kernel coefficient to the log-determinant  functional, see \cite{MR3194491}.

On noncommutative two tori, such local computation has been pushed to the $V_4$ term for the scalar Laplacian operator in  Connes-Farzard very recent preprint \cite{2016arXiv161109815C}. Calculation on the noncommutative four sphere will be performed in future project. 

\par

We conclude the introduction with an outline of the paper. In Section \ref{sec:mfld-notionsforCLdeformation}, we set up the basic notations for our noncommutative manifolds. The main results are recorded in  Section \ref{sec:scalcur-modscalcur} and \ref{sec:scalcur-EHaction}. In section \ref{sec:widomcal}, we briefly review  Widom's pseudo differential calculus, compare with the previous work \cite{Liu:2015aa}, the calculus is extended to vector bundles. All the notations extend naturally to
the noncommutative setting by the deformation machinery discussed in Section \ref{sec:mfld-notionsforCLdeformation}. On the computational side, the whole algorithm has been explained in \cite{Liu:2015aa} in great detail. What is new in Section \ref{sec:computation} is in the computation for the local curvature functions in all even dimensions.

%
%



\section{Notations for Connes-Landi deformations}
\label{sec:mfld-notionsforCLdeformation}

The class of noncommutative manifolds studied in this paper is obtained by deforming Riemannian manifolds along a torus action. The  spectral triple $(C^\infty(M_\theta), L^2(\slashed S), \slashed D)$  was originally proposed in  Connes-Landi \cite{MR1846904} and further investigated in \cite{MR2403605,MR1937657}.  The analytic aspect of the
deformation theory (known as $\theta$-deformation) was developed even earlier by Rieffel \cite{MR1184061}. One of  the reasons the author started to look at those
examples is the gauge theory aspects studied in  \cite{MR3262521}, in which the deformed algebra  $C^\infty(M_\theta)$ is called a toric noncommutative manifold. The author also used the name in the previous work \cite{Liu:2015aa}. 
Other related work on this subject are \cite{MR2230348, 2014arXiv1408.4429C, 2010LMaPh..94..263Y}.  
\par

In this section, we will follow the notations  in \cite{MR3262521} and \cite{MR2349630}. 


\subsection{Deformation along a $\T^n$-action}
Let $M$ be a closed (compact without boundary) Riemannian manifolds with a torus action as isometries: $\T^n\subset \mathrm{Iso}(M)$. Then $C^\infty(M)$ becomes a smooth representation of $\T^n$ via the pull-back action:
\begin{align}
    U_t(f)(x) = f(t^{-1} \cdot x), \,\,\, x\in M, \,\, t\in \T^n, \,\, f\in C^\infty(M).
    \label{eq:mfld-pullbackoncinfM}
\end{align}
From Fourier theory on torus, $C^\infty(M)$ admits a $\Z^n$-grading, namely any element $f$ has a isotypical decomposition:
\begin{align}
    \label{eq:mfld-isotypicaldecom}
    f = \sum_{r \in \Z^n} f_r, \,\,\, \text{with} \,\, f_r = \int_{\T^n} U_t(f) e^{-2\pi i r.t} dt.
\end{align}
Notice that $f_r$ are eigenvectors for the torus action:  $U_t(f_r) = e^{2\pi i r.t} f_r$, $t\in \T^n$.
The pointwise multiplication in $C^\infty(M)$ can be written as a convolution:
\begin{align}
    \label{eq:mfld-pwprod}
    fg = \sum_{r,l} f_r g_l, \,\,\,\text{with} \,\,f = \sum_r f_r, \,\,g = \sum_l g_r.
\end{align}
To start the deformation, we fix a $n \times n$ skew-symmetric matrix $\theta$, denote by $\chi_\theta(r,l) = \exp(i\pi \abrac{\theta r,l})$ where $r,l\in \Z^n$ a bi-character on $\Z^n$. The deformed multiplication $\times_\theta: C^\infty(M) \times C^\infty(M) \rightarrow C^\infty(M)$ is defined by twisting the convolution in \eqref{eq:mfld-pwprod} above by  $\chi_\theta$:    
\begin{align}
    \label{eq:mfld-deformedprod}
      f \times_\theta g = \sum_{r,l} \chi_\theta(r,l) f_r g_l.
\end{align}
We obtain a  new algebra $C^\infty(M_\theta)\defeq ( C^\infty(M), \times_\theta)$, by keeping the underlying topological vector space and replacing the pointwise multiplication for functions by the deformed one $\times_\theta$. 
\par

To obtain a spectral triple, we require $M$ to be a spin manifold, denote the spinor bundle by $\slashed S$ and the Dirac operator $\slashed D$ acting on the Hilbert space $\mathcal H = L^2(\slashed S)$ of $L^2$-sections of $\slashed S$. The torus action can be lifted to $\slashed S$ upto a double cover: $c:\tilde \T^n\rightarrow \T^n$. The isometric assumption of the action implies that $\mathcal H$ is a unitary representation of $\T^n$.
Functions on $M$ are bounded operators on $\mathcal H$ via left multiplication: $f \mapsto L_f$, $\forall f \in C^\infty(M)$. Observe that 
\begin{align*}
    L_{U_t(f)} = U_{\tilde t} L_f U_{-\tilde t}, \,\,\, \forall t = c(\tilde t) \in \T^n,\,\, \tilde t \in \tilde \T^n.
\end{align*}
We can extend the adjoint on the right hand side to all bounded operators on $\mathcal H$:
\begin{align*}
    T \mapsto U_{\tilde t} T U_{-\tilde t}, \,\,\, T \in B(\mathcal H), \,\, \tilde t \in \tilde \T^n,
\end{align*}
which make $B(\mathcal H)$ into a $\T^n$-module. Same deformation (as in \eqref{eq:mfld-deformedprod}) applies to operators, namely for any $\T^n$-smooth operator $T$ and vector $v \in \mathcal H$, 
\begin{align}
    \label{eq:tormfld-defn-deforops}
    \pi^\theta(T)(v) \defeq \sum_{r,l} \chi_\theta(r,l) T_r(v_l) 
\end{align}
defines a deformed operator $\pi^\theta(T)$. The associativity of the $\times_\theta$-multiplication turns out to be a functorial property for the deformation, whose categorical description can be found in \cite{MR3262521}. As a special case, we see that the whole representation $C^\infty(M) \subset B(\mathcal H)$ is deformed in the following sense. For two operators $T_1$ and $T_2$, let $T_1 \circ_\theta T_2$ be the deformed composition according to \eqref{eq:mfld-deformedprod}, we have
\begin{align*}
    \pi^\theta \brac{ 
    T_1 \circ_\theta T_2
    } = \pi^\theta(T_1) \pi^\theta(T_2).
\end{align*}


\par
It was shown in \cite{MR1846904} that the triple $(C^\infty(M_\theta), L^2(\slashed S), \slashed D)$ satisfies all axioms of a noncommutative spin geometry. Motivated by the functorial property, in the previous work \cite{Liu:2015aa}, we extend the deformation map $\pi^\theta$ to all pseudo differential operators and the multiplication $\times_\theta$ to all tensor calculus. The combination of two ingredients under the framework of Widom's work \cite{MR560744,MR538027} gives rise to the pseudo differential
calculus for the noncommutative manifolds.

\par
\subsection{Deformation of Riemannian geometry}
\label{subsec:notationsforlater}
The deformation extends to toric equivariant vector bundles (cf. \cite{MR3262521}) over $M$. It simply means that we can deform the $C^\infty(M)$-module structure on the smooth sections to  $C^\infty(M_\theta)$-bimodules using the same formula \eqref{eq:mfld-deformedprod}. For instance, the left action of $C^\infty(M_\theta)$ is defined in \eqref{eq:mfld-deformedprod} by taking $f$ being a function while $g$ being a section.  Moreover, we can deform the whole tensor calculus. Again,
in  \eqref{eq:mfld-deformedprod}, we can set $f$ and $g$ to be two vector fields $X$ and $Y$, then replace  the multiplication by  the fiberwise tensor product $\otimes$, then the right hand side of \eqref{eq:mfld-deformedprod} defines the deformed tensor product $ X \otimes_\theta Y$. For one-form $\omega$,  the pointwise contraction $X \cdot \omega \rightarrow X \cdot_\theta \omega$ is deformed in the same fashion. They share a mixed associativity:
\begin{align*}
    (X \otimes_\theta Y) \cdot_\theta \omega = X \otimes_\theta (Y \cdot_\theta \omega).
\end{align*}
Let $\nabla$ be the Levi-Civita connection, in fact, a toric-equivariant one
will work. What we need is the fact that the connection preserve the isotypic  components \eqref{eq:mfld-isotypicaldecom}, for example, for a smooth function,   $\nabla f_r = (\nabla f)_r$, where $r \in \Z^n$. The upshot is the Leibniz property, for instance, 
\begin{align*}
   \nabla (X \otimes_\theta Y) =  (\nabla X) \otimes_\theta Y + X \otimes_\theta (\nabla Y).
\end{align*}
For the rest of the paper,  it suffices to just accept the existence of the deformation, the explicit formula \eqref{eq:mfld-deformedprod} no longer matters. We will always use the deformed operations and skip notations like $\times_\theta$, $\pi^\theta(\cdot)$ and $\otimes_\theta$. We conclude this section with a few remarks about the notations in our main result Theorem \ref{thm:scalcur-RD_h}.
\begin{enumerate}
    \item For a function $h \in C^\infty(M_\theta)$, the exponential $k = e^h \in C^\infty(M_\theta)$ is a smooth function on $M$, because $C^\infty(M_\theta)= C^\infty(M)$ as a vector space, which has two interpretations. As an operator, we have $\pi^\theta(k) = \exp(\pi^\theta(h))$. In terms of a function, $k$ is obtained by applying the exponential power series to $h$ with respect to  $\times_\theta$.    
    \item The space of one-forms $\Gamma(T^*M)$ becomes a bimodule over $C^\infty(M_\theta)$ after deformation, the differences can be seen from some elementary examples, for a smooth function $h$
    \begin{align*}
        \nabla (h^2) = h \nabla h + (\nabla h) h \neq 2 h (\nabla h) \,\, \text{or $2(\nabla h) h$}.
    \end{align*}
    What is more interesting is the formula for $\nabla e^h$, see \eqref{eq:introv2-f1}. 
\item Functions or sections which are invariant under the torus become central elements after deformation, examples include  the scalar curvature function $\mathcal S$, metric tensor $g^{-1}$ on the cotangent bundle, etc. Indeed, in \eqref{eq:mfld-deformedprod}, if one of the functions is invariant under the torus action, then the deformed product is reduced to the commutative version.  In particular, in \eqref{eq:sclcur-RDlogk}, we have $(\nabla^2 h)\cdot_\theta g^{-1} = (\nabla^2
    h) g^{-1}= -\Delta h$, for the other term $(\nabla h)(\nabla h)g^{-1}$, it means   
    $$(\nabla h) \otimes_\theta (\nabla h)\cdot_\theta g^{-1} = 
      (\nabla h) \otimes_\theta (\nabla h) \cdot g^{-1},
    $$ 
    which is sort of a deformed inner product $\abrac{\nabla h, \nabla h}_\theta$. On noncommutative two tori, it equals the Dirichlet quadratic form $\sum_{j=1}^1 (\delta_j h)^2$ defined in \cite{MR3194491}. 
\end{enumerate}

\section{Conformal geometry and the scalar curvature }
\label{sec:scalcur-modscalcur}

When proposing conformal geometry in the noncommutative setting, the first criterion for us is the following: the new version should include the Riemannian counterpart as a special case. Thanks to the underlying spin structure, the construction is more straightforward than the one in \cite{MR3194491}, which is a special case explored in \cite{MR2427588}. The basic notions of the conformal geometry of Connes-Landi spectral triples are summarize in Table \ref{tab:defn-curvature} in the introduction.

%
\subsection{Conformal change of metric}
\label{sec:scalcur-con-ch-met} 
Taking advantage of the existence of the Spin structure in our setting, conformal change of metric for Connes-Landi spectral triple is a special case of perturbed spectral triple in \cite{MR2427588}, Section 2.2. Let us quickly review the motivation explained in \cite{MR2427588}. On a Spin manifold $M$, let $\slashed S^g$ be the spinor bundle associated to $g$. For a conformal change of metric $g' = e^{-2h} g$, with $h \in C^\infty(M)$ self-adjoint, there exists a Spin-equivariant gauge transformation $\beta^g_{g'}$ sending $g$-spinorial frames to $g'$-spinorial frames. The gauge transformed Dirac operator $^g\slashed D^{g'} \defeq
\beta_g^{g'} \slashed D^{g'} \beta^g_{g'}$ was computed in \cite{MR1158762}:
\begin{align}
    ^g\slashed D^{g'}  = e^{(m+1)h/2}\slashed D^{g} e^{(1-m)h/2}, \,\,\, \text{where $m=\dim M$.}
    \label{eq:scalcur-gaugedDirac}
\end{align}
In order to get isometries at the level of $L^2$-sections, we have to account for the change of the Riemannian volume form $\op{Vol}_{g'} = e^{-mh} \op{Vol}_{g}$ and  rescale the gauge transformations $\beta^g_{g'}$
\begin{align*}
    \tilde \beta^g_{g'} \defeq e^{mh/2}\circ \beta^g_{g'} = \beta^g_{g'} \circ e^{mh/2} 
\end{align*}
from $L^2(\slashed S^g)$ to $L^2(\slashed S^{g'})$. Finally, the conformal change of metric leads to the following perturbation:
\begin{align*}
    \slashed D \mapsto D_h \defeq e^{h/2}\slashed D e^{h/2}.
\end{align*}

Given a Connes-Landi deformation $(C^\infty(M_\theta), L^2(\slashed S),\slashed D)$, the conformal change of metric is implemented  by a perturbation of the Dirac operator: $\slashed D \mapsto \slashed D_h$, 
where  $D_h \defeq k^{1/2} \slashed D k^{1/2}$, where $k = e^h$ and $h \in C^\infty(M_\theta)$ self-adjoint. Denote the modular operator 
\begin{align*}
    \modop(x) \defeq e^{-2h} x e^{2h}, \,\, \forall x \in C^\infty(M_\theta).
\end{align*}
It was shown in \cite{MR2427588} that the twisted commutator $[D_h,\modop^{-1/2}(x)]$ is bounded for all $x \in C^\infty(M_\theta)$. 
In particular, we have 
\begin{prop}
    The perturbed spectral triple $(C^\infty(M_\theta), L^2(\slashed S), D_h)$ is a  $\modop^{-1/2}$-twisted \textup{(}also called a modular\textup{)} spectral triple in the sense of \emph{Definition 3.1} of  \cite{MR2427588}. 
\end{prop}

\subsection{Change of volume form}
For any section $\psi \in \Gamma(\End(\slashed S))$, we denote by $\varphi_0$ the trace functional with respect to the metric $g$
\begin{align}
    \label{eq:scalcur-vol-phi0}
    \varphi_0(\psi) = \int_M \Tr_x (\psi)d\op{vol}_g,
\end{align}
where $\Tr_x$ is the fiberwise trace. After the conformal change of metric defined in the previous section, the new volume form is given by the non-tracial weight:
\begin{align}
    \label{eq:scalcur-nontracial-vol}
    \varphi_h(\psi) \defeq \varphi_0( \psi e^{-mh}).  
\end{align}
Both $\varphi_0$ and $\varphi_h$ have an intrinsic definition via Wodzicki residue. 
\begin{defn}[Wodzicki residue]
    \label{defn:scalcur-wodres}
    Let $P$ be a deformed pseudo differential operator acting on the spinor sections $\Gamma(\slashed S)$, we define
    \begin{align}
        \label{eq:scalcur-wodres}
        \dashint P \defeq \op{Res}_{z=0}\Tr \left( P \slashed D^{-z} \right).  
    \end{align}
\end{defn}
Upto a constant factor, we have for any $\psi \in C^\infty(M_\theta)$:
\begin{align*}
    \varphi_0(\psi) =  \dashint \psi \slashed D^{-m}, \,\,\,\,
    \varphi_h(\psi) = \dashint  \psi D_h^{-m}.
\end{align*}

\subsection{Local curvature functions}
\label{subsec:localcurfun}

The following family of functions appears in the calculation in section \ref{sec:computation}. In this subsection, we assume that $s,t>0$, $u\leq 0$, $p,q,l \in \Z_{\ge 0}$ and $m = \dim M$ as usual. Denote: 
\begin{align}
    \begin{split}
    \tilde K_{(p,q)}(u,s) &= (1-u)^{-p} (s-u)^{-q}, \\
    \tilde H_{(p,q,l)}(u,s,t) &= (1-u)^{-p} (s-u)^{-q} (st-u)^{-l}.
    \end{split}
    \label{eq:scalcur-Kpq-Hpql}    
\end{align}
 If one substitutes the functions defined above into Figure \ref{fig:onevarfun} in appendix \ref{app:completeb2terrm}, for example:
\begin{align*}
    \mathtt{  Kmod[2,1, s, dimM]} \rightarrow \tilde K_{(2,1)}(u,s),\,\,\,
\mathtt{  H[2,1,1, s, t, dimM]} \rightarrow  \tilde H_{(2,1,1)}(u,s,t),
\end{align*}
and then take the summation, the result is the following function:
\begin{align}
    F(u,s,m) = \frac{m \left(\sqrt{\mathit{s}}+1\right) (u-1)+2 \left(2 \sqrt{\mathit{s}}+u+1\right)}{m (u-1)^3 (u-\mathit{s})}.
    \label{eq:scalcur-Ktotal}
\end{align}
Similarly, to get the two-variable function for the modular scalar curvature, one has to replace the terms in Figure \ref{fig:twovarfunp1} and \ref{fig:twovarfunp2} in  appendix \ref{app:completeb2terrm} by  $\tilde K_{(p,q)}$ and $\tilde H_{(p,q,l)}$ and sum together. We record the result: 
\begin{align}
    \begin{split}
   &\,\, G(u,s,t,m) \\
    =&\,\, \left\{
    \mathit{s}^{3/2} \left(m \left(\left(2 \sqrt{\mathit{s}}+1\right) \sqrt{\mathit{t}}+2\right)-4 \left(\left(2 \sqrt{\mathit{s}}+1\right) \sqrt{\mathit{t}}+1\right)\right)  \right.\\
    &\,\, \left.
     +u^2 \left(2 \left(\mathit{s}+2 \sqrt{\mathit{s}}+3\right) \left(\sqrt{\mathit{s}} \sqrt{\mathit{t}}+1\right)-m \left(\mathit{s}^{3/2} \sqrt{\mathit{t}}-2 \mathit{s} \left(\sqrt{\mathit{t}}-1\right)-2 \sqrt{\mathit{s}}+2\right)\right)
\right. \\
&\,\, \left.
-u\sqrt{\mathit{s}}\left[\mathit{s} \left(6 \sqrt{\mathit{t}}+4\right)+\sqrt{\mathit{s}} \left(10-4 \sqrt{\mathit{t}}\right)-2 \left(\sqrt{\mathit{t}}+1\right)
\right. \right.\\
&\,\, \left. \left.
+m \left(2 \mathit{s}^{3/2} \sqrt{\mathit{t}}+2 \sqrt{\mathit{s}} \left(\sqrt{\mathit{t}}-1\right)+2 \mathit{s}+\sqrt{\mathit{t}}+2\right)\right]
\right. \\
&\,\, \left.
+u^3 \left(m \sqrt{\mathit{s}} \sqrt{\mathit{t}}+2 m+2 \sqrt{\mathit{s}}+2\right)
\right\} \left( m (u-1)^3 (u-\mathit{s})^2 (u-\mathit{s} \mathit{t}) \right)^{-1}.
    \end{split}
       \label{eq:scalcur-Htotal}.
\end{align}

When $m =\dim M \ge 4$ and is even, the local curvature functions in theorem \ref{thm:scalcur-curforDeltak} are determined by the germs of $F$ and $G$ at $u=0$, let $\tilde j_0 = (m-4)/2$:
\begin{align}
    K^{(m)}_{\Delta_k}(s) &= \tilde j_0! \frac{d^{\tilde j_0}}{du^{\tilde j_0}}\Big|_{u=0} F(u,s,m),
    \label{eq:scalcur-KDeltak} \\
 H^{(m)}_{\Delta_k}(s,t) &=
\tilde j_0! \frac{d^{\tilde j_0}}{du^{\tilde j_0}}\Big|_{u=0} G(u,s,t,m).
\label{eq:scalcur-HDeltak}
\end{align}
The appearance of the differential $\frac{d^{\tilde j_0}}{du^{\tilde j_0}}$ is due to Lemma \ref{lem:modcur-put-lamda-to-1}. The proof 
can be found in Section \ref{subsec:comformodcur}. 
\par
As an example, we compute in dimension four, in which $\tilde j_0 =0$: 
\begin{align*}
    K^{(4)}_{\Delta_k}(s) = F(0,s,4) =  -(2s)^{-1},\,\,\,\, 
    H^{(4)}_{\Delta_k}(s) = G(0,s,t,4) = (s^{3/2} t)^{-1}.
\end{align*}
The length of the expressions grows fast when the dimension increases, the explicit expressions in dimension 6 and 8 are listed at the end of section \ref{subsec:comformodcur}.
\subsection{The modular scalar curvature }
\label{subsec:scalcur-modscal}
In the rest of the paper, $M$ will be a even dimensional closed spin manifold which admits a Connes-Landi deformation $M_\theta$, where $\theta$ is a skew symmetric matrix. Let $h=h^* \in C^\infty(M_\theta)$ be a log-Weyl factor, then the perturbation $\slashed D \rightarrow D_h = e^{h/2} \slashed D  e^{h/2}$ represents the conformal change of metric $g\rightarrow g' = e^{-2h}g$ in the Riemannian situation. The following heat kernel asymptotic has been established in \cite{Liu:2015aa}: $\forall f \in C^\infty(M_\theta)$,
    \begin{align}
        \label{eq:scalcur-heatasymp}
        \Tr\brac{f e^{-tD_h^2}} \backsim_{t\searrow 0} \sum_{j=0}^\infty V_j(f,D^2_h) t^{(j-m)/2}, \,\,\, m=\dim M,  
    \end{align}
    whose coefficients are functionals on $C^\infty(M_\theta)$.
\begin{defn}
    \label{defn:scalcur-RDh}
     We define the modular scalar curvature $ R_{D_h} \in C^\infty(M_\theta)$ associated to the metric $D_h$ to be the action density of the second coefficient: 
    \begin{align}
        \label{eq:scalcur-defn-RD_h}
        V_2(f,D^2_h) = \varphi_0 \left( f R_{D_h} \right) ,\,\,\, \forall f \in C^\infty(M_\theta),
    \end{align}
    where the $\varphi_0$ is defined in \eqref{eq:scalcur-vol-phi0}.
\end{defn}
\begin{rem}
    \hspace{1cm}
    \begin{enumerate}
        \item To recover the formula for conformal change of metric in Riemannian geometry, one should consider the action density with respect to the new volume functional $\varphi_h$ in \eqref{eq:scalcur-nontracial-vol}. Nevertheless, the difference of the two definitions is only a volume factor $e^{ -mh }$. We will stick with the definition above because we will need the expression of $R_{D_h}$ defined above which contains the volume factor $e^{ -mh }$.
        \item Provided the heat kernel asymptotic  \eqref{eq:scalcur-heatasymp}, one can quickly show that the Wodzicki residue in Definition \ref{defn:scalcur-wodres} can be calculated through the perturbed Dirac $D_h$, namely
            \begin{align}
                \dashint P  = \op{Res}_{z=0}\Tr \left( P D_h^{-z} \right).
                \label{eq:scalcur-wodres-D_h}
            \end{align}
By the basic correspondence between the heat kernel asymptotic and the meromorphic extension of the zeta function (cf. \cite{connes2008noncommutative}, Section 11), we have, up to a constant factor:
\begin{align}
    V_2(f,D^2_h) = \dashint f D_h^{-m+2}, \,\,\, \forall f \in C^\infty(M_\theta), \,\,\, m =\dim M.
    \label{eq:scalcur-compare-V_2-dashint}
\end{align}
 The right hand side is the intrinsic definition for the scalar curvature in  \cite{connes2008noncommutative} Definition 11.1.    
From the point of view, the conformal geometry for the Connes-Landi spectral triple is an example of noncommutative mass scale (cf. \cite{MR2239979}) in the corresponding spectral action. 
 \end{enumerate}
\end{rem}
To simplify the calculation, we start with the perturbed Laplacian  $\Delta_k$ below, which equals $D_k^2$ up to a conjugation, where $k=e^h$ is the Weyl factor. Namely,
\begin{align}
    \label{eq:scacur-defn-Deltak}
    \Delta_k = k^{1/2}  D_k^2 k^{-1/2} = k^2 \slashed D + k[\slashed D, k] \slashed D
    = k^2 \slashed D + k \mathbf c(\nabla k) \slashed D,
\end{align}
where $\mathbf c(\nabla k)$ denotes the deformed Clifford action. We shall see later that the leading symbol of $\Delta_k$ is equal to $k^2 \abs\xi^2 \in C^\infty(T^*M_\theta)$, where $\abs\xi^2$ is the squared length function with respect to the non-deformed Riemannian metric. Therefore we consider the modular operator $\modop$ and its logarithm $\logmodop$ (called modular derivation):
\begin{align}
    \label{eq:scalcur-modop-logmodop}
    \modop(\psi) \defeq e^{-2h} \psi e^{2h},\,\,\, \logmodop(\psi)\defeq -2\op{ad}_h(\psi) = 2(\psi h - h\psi),
\end{align}
here $\psi$ can be any vector on which $C^\infty(M_\theta)$ acts from two sides, such as a tensor field  or a spinor section over $M_\theta$.  \par
 
Let $m=\dim M$ even,  denote the volume of the unit sphere by $$\mathrm{Vol}(\mathbb S^{m-1}) = (2 \pi^{m/2})/\Gamma(m/2),$$ we need the following constant in the theorems in this section: 
\begin{align}
    \frac{1}{(2\pi)^m}\cdot \frac12 \cdot \mathrm{Vol}(\mathbb S^{m-1}) = \frac{1}{(4\pi)^{m/2}} \frac{1}{\Gamma(m/2)}. 
    \label{eq:scalcur-overallcont}
\end{align}
    
Before stating the main result, we recall the classical result for the scalar curvature $\mathcal S$ with respect to a conformal change of metric $g' =e^{-2h} g $, $m=\dim M$ (cf. \cite{MR867684}, Theorem 1.159):
\begin{align}
    \label{eq:scalcur-chofscalcur-commutative}
   \mathcal S_{g'} =   e^{2h} \brac{
             2(m-1) \Delta h - (m-2)(m-1)\abrac{d h ,d h}_{g^{-1}} +\mathcal S_{g}}, 
\end{align}
where $g^{-1}$ is the metric tensor on the cotangent bundle, $\Delta h$ is the Laplacian of $h$ and $\abrac{d h ,d h}_{g^{-1}} = \abs{df}^2$ is the squared length of the one-form $dh$ with respect to the starting metric $g$.

\begin{thm}
    \label{thm:scalcur-curforDeltak}
    Let $M_\theta$ be a toric noncommutative manifold \textup (or a Connes-Landi deformation\textup ) with $m = \dim M$ even. For $k =e^h$, let $\Delta_k$ be the perturbed Laplacian defined in \eqref{eq:scacur-defn-Deltak} and $R_{\Delta_k}(k)$ be the action density function  for  $V_2(\cdot,\Delta_k)$ in \eqref{eq:scalcur-heatasymp}  with respect to $\varphi_0$ in \eqref{eq:scalcur-vol-phi0}, that is 
    \begin{align*}
       V_2(f ,\Delta_k) = \varphi_0(fR_{\Delta_k}(k)), \,\,\, \forall f\in C^\infty(M_\theta).
    \end{align*}
    Up to the overall constant  given in \eqref{eq:scalcur-overallcont}, we have:
\begin{align}
\label{eq:scalcur-R-Deltak}
R_{\Delta_k}(k) &= k^{-m+1} K^{(m)}_{\Delta_k}(\modop)(\nabla^2 k) \cdot g^{-1}
    + k^{-m} H^{(m)}_{\Delta_k}(\modop_{(1)},\modop_{(2)}) \brac{
    \nabla k\nabla k
}\cdot g^{-1} \\
&\,\,\, \,\, + c^{(m)}_{\Delta} k^{-m+2} \mathcal S_\Delta. \nonumber
\end{align}
The local curvature functions $K^{(m)}_{\Delta_k}$ and $H^{(m)}_{\Delta_k}$ are given in equation \eqref{eq:scalcur-KDeltak} and  \eqref{eq:scalcur-HDeltak} respectively. 
\end{thm}
The formula \eqref{eq:scalcur-R-Deltak} contains all the necessary ingredients in the Riemannian version: $\mathcal S_\Delta$ is the scalar curvature function, $g^{-1}$ stands for the dual  metric tensor on $T^*M$ and $\nabla$ is the Levi-Civita connection. The constant $c^{(m)}_{\Delta}$ is defined in \eqref{eq:modcur-constcm}. Internal operations, like tensor product or contractions, between functions and tensor fields  have been discussed at the end of   Section \ref{subsec:notationsforlater}.  


%
\par

The quantum part of the curvature is given by the action of the modular operator $\modop$ (in \eqref{eq:scalcur-modop-logmodop}) through two local curvature functions $K^{(m)}_{\Delta_k}$ and $H^{(m)}_{\Delta_k}$ discussed in the previous section \ref{subsec:localcurfun}, see equation \eqref{eq:scalcur-KDeltak} and  \eqref{eq:scalcur-HDeltak}. \par

Now we are ready to compute the expression of $R_{D_h}$ in Definition \eqref{defn:scalcur-RDh}.
Notice that for any $f \in C^\infty(M_\theta)$,
\begin{align*}
    \Tr \brac{f e^{-t  D^2_h} } = \Tr \brac{k^{1/2} f e^{-t  D^2_h} k^{-1/2}}
    = \Tr \brac{k^{1/2} f k^{-1/2} e^{-t \Delta_h}},
\end{align*}
hence  
\begin{align*}
    V_2(f,  D_k^2) = V_2(k^{1/2} f k^{-1/2}, \Delta_h) = \int_M k^{1/2} f k^{-1/2} R_{\Delta_h}
    = \int_M  f k^{-1/2} R_{\Delta_h} k^{1/2}.
\end{align*}

Therefore the scalar curvature densities for $\slashed D^2_k$ and $\Delta_k$ are related via
\begin{align}
    \label{eq:scalcur-RDvsRDelta}
    R_{D_h} = k^{-1/2} R_{\Delta_k} k^{1/2}.
\end{align}

\begin{cor}
    \label{cor:scalcur-KHDelta_k}
    Keep the notations. The modular scalar curvature $R_{D_h}$ for the metric $D_h = e^{h/2} \slashed D e^{h/2}$ shares the same form as $R_{\Delta_k}$ defined in \eqref{eq:scalcur-R-Deltak} with the new local curvature functions given by
\begin{align*}
    K^{(m)}_{ D^2_k}(s) =s^{1/4} K^{(m)}_{\Delta_k}(s), \,\,\,
H^{(m)}_{ D^2_k}(s,t)= s^{1/4}t^{1/4} H^{(m)}_{ \Delta_k}(s,t)
\end{align*}
with $s,t>0$.
\end{cor}

\subsection{Modular scalar curvature in terms of $h = \log k$}
\label{subsec:scalcur-h-modscal}
For $k = e^h$, we have, similar to \cite{MR3194491} Sec. 6.1 and \cite{leschdivideddifference} Example 3.13, 
\begin{align}
    \label{eq:scalcur-nablak}
    \nabla k &= k  f_1(\logmodop/2)(\nabla h), \\ 
 \nabla^2 k &= k \brac{ f_{1} (\logmodop/2)(\nabla^2 h) + 2g_{2}(\logmodop_{(1)}/2, \logmodop_{(2)}/2)(\nabla h \nabla h)}, 
\label{eq:scalcur-nabla^2k}
\end{align}
where  $\logmodop = -2\op{ad}_h$ is the modular derivation and the functions $f_1$  and $g_2$ are given by
\begin{align}
    \label{eq:scalcur-f1g2}
    f_1(s) &= \frac{e^s-1}{s}, \\
    g_2(s,t) &=\frac{f_1(s+t) - f_1(s)}{t}=
    \frac{s e^{s+t}-e^s (s+t)+t}{s t (s+t)}.
\label{eq:scalcur-g2}
\end{align}
 We remark that $1/f_1(s) = s/(e^s -1)$ is the  generating function of Bernoulli  numbers. Equation \eqref{eq:scalcur-nablak} can be viewed and a generalization of the chain rule in calculus. One can compare it with the differential of the exponential map of Lie groups whose Jacobian involves the same function, for instance, Prop. 5.1 in \cite{berline1992heat}.  

After the substitution, $R_{\Delta_k}$ in \eqref{eq:scalcur-R-Deltak} becomes
\begin{align}
    \label{eq:scacur-RDeltalogk}
    R_{\Delta_k}(h) &= \,k^{-m+2} \brac{
        \mathcal K^{(m)}_{\Delta_k}(\logmodop) (\nabla^2 h) + 
    \mathcal H^{(m)}_{\Delta_k}(\logmodop_{(1)}, \logmodop_{(2)})(\nabla h \nabla h)
} g^{-1}  \\
&+c_{\Delta} k^{-m+2} \mathcal S_\Delta \nonumber.
\end{align}
\begin{prop}
    The modular curvature functions appeared in \eqref{eq:scalcur-R-Deltak} and \eqref{eq:scacur-RDeltalogk} are related by
\begin{align}
 \label{eq:scacur-relation-RDeltak-RDeltalogk}
 \begin{split}
   \mathcal K^{(m)}_{\Delta_k}(s) &=  K^{(m)}_{\Delta_k}(e^s) f_1(s/2), \\
    \mathcal H^{(m)}_{\Delta_k}(s,t) &= e^{s/2}  H^{(m)}_{\Delta_k}(e^s,e^t) f_1(s/2)f_1(t/2)
    +2 K_{\Delta_k}(e^{s+t}) g_2(s/2,t/2)  
 \end{split}
\end{align}
where $s,t \in \R$.
\end{prop}

Recall \eqref{eq:scalcur-RDvsRDelta} which states that $R_{D_h}$ and $R_{ \Delta_k}$ differ by a conjugation by $k^{1/2}$.
We finally achieve the complete formula for the modular scalar curvature in terms of the log-Weyl factor $h$ and the associated modular derivative $\logmodop = -2\op{ad}_h$.

 \begin{thm}[Modular scalar curvature]
     \label{thm:scalcur-RD_h}
    For the perturbed Dirac operator $D_h = e^{h/2} \slashed D e^{h/2}$ which represents a noncommutative metric that is conformal to the original one, we have the following conformal change of scalar curvature: upto a overall constant defined in \eqref{eq:scalcur-overallcont}, 
    the associated scalar curvature density in \textup{Definition  \ref{defn:scalcur-RDh}} is given by:
    \begin{align}
            \label{eq:sclcur-RDlogk}
        R_{ D_h} &= \, e^{(-m+2)h} \brac{
            \mathcal K^{(m)}_{ D_h}(\logmodop) (\nabla^2 h) + 
    \mathcal H^{(m)}_{ D_h}(\logmodop_{(1)}, \logmodop_{(2)})(\nabla h \nabla h)
} g^{-1}  \\
&+c_{\Delta_k} e^{(-m+2)h} \mathcal S_\Delta , \,\,\, m=\dim M\nonumber,
    \end{align}
    with:
    \begin{align*}
        \mathcal K^{(m)}_{ D_h}(s) & = e^{s/4} K^{(m)}_{\Delta_k}(e^s) f_1(s/2), \\
       \mathcal H^{(m)}_{ D_h}(s,t) &= e^{s/4}e^{t/4} 
       \left( e^{s/2}  H^{(m)}_{\Delta_k}(e^s,e^t) f_1(s/2)f_1(t/2) \right. \\ 
       &+ \left. 2 g_2(s/2,t/2)  K^{(m)}_{\Delta_k}(e^{s+t}) 
\right).   
\end{align*}
The functions $f_1$ and $g_2$ are given in \eqref{eq:scalcur-f1g2} and \eqref{eq:scalcur-g2}, while for  $K_{\Delta_k}$ and $H_{\Delta_k}$, see equation \eqref{eq:scalcur-KDeltak} and  \eqref{eq:scalcur-HDeltak}. 
\end{thm}
Notations in \eqref{eq:sclcur-RDlogk} and \eqref{eq:scalcur-R-Deltak} are quite similar, thus we do not repeat the explanation again. However, it  worth mentioning that equation \eqref{eq:sclcur-RDlogk} gives a more transparent  comparison with the formula in Riemannian geometry \eqref{eq:scalcur-chofscalcur-commutative}. Indeed, one can get rid of the volume factor $e^{-mh}$ by 
using the new volume form $\varphi_h$ in \eqref{eq:scalcur-nontracial-vol} to define the action density instead of $\varphi_0$ in Definition \ref{defn:scalcur-RDh}. Also, observe that in the commutative setting $\logmodop = 2\op{ad}_h =0$, as a result, $\mathcal K_{\slashed D_h}(\logmodop) (\nabla^2 h) =\mathcal K_{\slashed D_h}(0) (\nabla^2 h)$ and $\mathcal H_{\slashed D_h}(\logmodop_{(1)}, \logmodop_{(2)})(\nabla h \nabla h) = \mathcal H_{\slashed D_h}(0,0)(\nabla h \nabla h)$,
therefore \eqref{eq:sclcur-RDlogk} is reduced to \eqref{eq:scalcur-R-Deltak} completely. Informally speaking, the deformation of the scalar curvature is obtained by adding extra action of modular derivations onto the conformal change of the scalar curvature in Riemannian geometry.  

The explicit expressions for $\mathcal K^{(4)}_{\slashed D_h}$ and $\mathcal H^{(4)}_{\slashed D_h}$ are listed in \eqref{eq:scalcur-dim4KandH_forRD_h}.

\section{The Einstein-Hilbert action }
\label{sec:scalcur-EHaction}
In this section, we will perform  similar calculation as in \cite{MR3194491} section 4 to study the variation of the noncommutative analog  of the  Einstein-Hilbert action. More precisely, we shall compute the variation in two ways, as a result, we obtain a abstract proof of the functional relations between  $\mathcal K^{(m)}_{\slashed D_h}$ and $\mathcal H^{(m)}_{\slashed D_h}$. Since  $\mathcal K^{(m)}_{\slashed D_h}$ and $\mathcal H^{(m)}_{\slashed D_h}$ has been computed explicitly, a
straightforward verification of the relations provides a conceptual confirmation of our computation. On the other hand, such relations will be a important key to gain more conceptual understanding about the two-variable function  $\mathcal H^{(m)}_{\slashed D_h}$. 
We will first state the relations in  section \ref{sec:scalcur-intrelation} and the proof occupies the later section \ref{sec:varofEH}.

\subsection{Internal relations between $\mathcal K^{(m)}_{ D_h}$ and $\mathcal H^{(m)}_{ D_h}$}
\label{sec:scalcur-intrelation}
Let $m = \dim M$ and $j_0= (m-2)/2$. We define:
\begin{align}
    \label{eq:Keh-scalcur}
        K_{\op{EH}}(s) &= -8 j_0\frac{ \sinh(s/4)}{s} \mathcal K_{\slashed D_h}(s),\\
        H_{\op{EH}}(s,t) &=-8 j_0\frac{ \sinh((s+t)/4)}{s+t} \mathcal H_{\slashed D_h}(s,t).
    \label{eq:Heh-scalcur} 
    \end{align}
They are the  local curvature functions for the gradient of the Einstein-Hilbert action found in corollary \ref{cor:scalcur-KEHandHEH}. \par  
    To link $K_{\op{EH}}$ and $H_{\op{EH}}$, we need the following one variable functions:
\begin{align}
    T(s) &= 2j_0 \mathcal K_{\slashed D_h}(0) f_1(-j_0 s) + \mathcal H_{\slashed D_h}(s,-s), 
    \label{eq:scalcur-T} \\
    \tilde T(s) &= T(-s) e^{-s}
    \label{eq:scalcur-tildeT}
\end{align}
and two variable functions:
\begin{align}
  \label{eq:L-scalcur}
    L(s,t) &= -2 T(s) j_0  f_1(-j_0(s+t)), \\
   M(s,t) &= 2  e^{-j_0 (s+t)}
   \left[ \frac{T(-t) - T(s)}{s+t} + \frac{T(t) - T(-s)}{s+t} e^{j_0t} \right],
    \label{eq:M-scalcur}\\
\label{eq:P-scalcur}
   P(s,t) &=-2j_0  f_1(-j_0s) T(t) + (-2)\frac{T(s+t) - T(t)}{s} + 2\frac{T(s+t) - T(s)}{t}, \\
    Q(s,t) &=-2j_0  f_1(-j_0s) \tilde T(t)
    + (-2)\frac{\tilde T(s+t) - \tilde T(t)}{s} + 2\frac{\tilde T(s+t) -\tilde T(s)}{t}.
   \label{eq:Q-scalcur}
\end{align}
The geometric meaning of $T(s)$ is given in Prop. \ref{prop:scalcur-EHlocal}. 
In dimension four, that is $j_0 = 1$, all the functions $T, \tilde T, L, M, P, Q$ were computed in \cite{MR3369894}, Theorem 5.1. 


\begin{thm}\label{thm:funrel-scalcur}
    Keep the notations as above. The internal relations are given by
    \begin{align}
        \label{eq:scalcur-relone}
        K_{\op{EH}}(s) &=- (T+\tilde T)(s) , \\
        H_{\op{EH}}(s,t) &=  (L+M-P-Q)(s,t). 
        \label{eq:scalcur-reltwo}
    \end{align}
\end{thm}
\begin{proof}
    Based on the calculation in Section \ref{sec:varofEH}, the local curvature functions on the left hand side 
   appeared in Corollary \ref{cor:scalcur-KEHandHEH} while those on the right hand side are taken from Proposition \ref{prop:scalcur-varlocalEH}. 
  They  both represent the gradient of the Einstein-Hilbert action thus should be the same. 
\end{proof}
\begin{rem}
    The relations \eqref{eq:scalcur-relone} and \eqref{eq:scalcur-reltwo} has two significant applications: 
    \begin{enumerate}
        \item The explicit expressions of  $\mathcal K^{(m)}_{\slashed D_h}(s)$ and $\mathcal H^{(m)}_{\slashed D_h}(s,t)$ are obtained after enormous amount of computation. We obtain a strong conceptual confirmation for our calculation  
            by verifying they do satisfy \eqref{eq:scalcur-relone} and \eqref{eq:scalcur-reltwo}. 
        \item As explained in the introduction, there is some familiar pattern
           shown in $\mathcal K^{(m)}_{\slashed D_h}(s)$, but for the two-variable function $\mathcal H^{(m)}_{\slashed D_h}(s,t)$, we have very little conceptual understanding for the present. The relations \eqref{eq:scalcur-relone} and \eqref{eq:scalcur-reltwo} might provide a starting point when we know the one variable function better. 
    \end{enumerate}
\end{rem}

Let us perform a quick verification in dimension four for \eqref{eq:scalcur-relone} to illustrate the first remark above. When $m=\dim M =4$, $j_0 = (m-2)/2 = 1$, from Theorem \ref{thm:scalcur-RD_h}, we can compute explicitly: 
\begin{align}
    \begin{split}
         \mathcal K^{(4)}_{\slashed D_h}(u) & =-\frac{2 e^{-\frac{u}{2}} \sinh \left(\frac{u}{4}\right)}{u},
\\
          \mathcal H^{(4)}_{\slashed D_h}(u,v) & =
          \frac{4 e^{\frac{1}{4} (-3) (u+v)} \left(u \left(-e^{v/2}\right)+\left(e^{u/2}-1\right) e^{v/2} v+u\right)}{u v (u+v)}.
    \end{split}
    \label{eq:scalcur-dim4KandH_forRD_h}
\end{align}
In particular
\begin{align*}
    H^{(4)}_{\slashed D_h}(u,-u) = -\frac{-2 u-4 e^{-\frac{u}{2}}+4}{u^2}.
\end{align*}
Hence we can calculate $K_{\op{EH}}(u)$ from the right hand side of \eqref{eq:scalcur-relone}.
\begin{align}
    &\,\,T(u) + T(-u) e^{-u} \nonumber\\
=&\,\,    \frac{-2 u-4 e^{-\frac{u}{2}}+4}{u^2}-e^{-u} \left(-\frac{2 u-4 e^{u/2}+4}{u^2}-\frac{e^u-1}{u}\right)-\frac{e^{-u}-1}{u}.
\label{eq:scalcur-T+Tdim4}
\end{align}
On the other hand,  we compute the left hand side of \eqref{eq:scalcur-relone}  using \eqref{eq:Keh-scalcur}:
\begin{align}
    -8 j_0\frac{ \sinh(u/4)}{u} \mathcal K_{\slashed D_h}(u) =
   - \frac{8 \left(e^{-\frac{3 u}{4}}-e^{-\frac{u}{4}}\right) \sinh \left(\frac{u}{4}\right)}{u^2},
    \label{eq:scalcur-Kehdim}
\end{align}
Both of \eqref{eq:scalcur-T+Tdim4} and \eqref{eq:scalcur-Kehdim} can be simplified to $-4 e^{-u}u^{-2} \left(e^{u/2}-1\right)^2$. \par
Interested readers are encouraged to perform similar verification for relation \eqref{eq:scalcur-reltwo}. 

\subsection{Variation of the Einstein-Hilbert action}
\label{sec:varofEH}
If we think of the spectral triple $(C^\infty(M_\theta),L^2(\slashed S), \slashed D)$ as an noncommutative Riemannian manifold, a conformal change of metric is implemented by a twisted version $(C^\infty(M_\theta),L^2(\slashed S), D_h)$ with $D_h = e^{h/2}\slashed D e^{h/2}$ and $h=h^*\in C^\infty(M_\theta)$. 
Then the straightforward analogue of the Einstein-Hilbert action (as a functional in $h$) is given by
\begin{align}
    \begin{split}
        F_{\op{EH}}(h) &\defeq V_2(1, D_h^2) =  \varphi_0(R_{\slashed D_h})  =
   \varphi_0 \left( e^{(2-m)h} 
    \mathcal K_{ D_h}(\logmodop) (\nabla^2 h) g^{-1} \right. \\ &\,\,\, \left.+ 
e^{(2-m)h} \mathcal H_{ D_h}(\logmodop_{(1)}, \logmodop_{(2)})(\nabla h \nabla h) g^{-1}
+ e^{(2-m)h} c_{\Delta_h}(m) \mathcal S_\Delta
\right).     
    \end{split}
    \label{eq:scalcur-EH-defn}
\end{align}
Let us make a remark about the comparison to the commutative situation. In Riemannian geometry, the Einstein-Hilbert action is defined as $F(g) = \int \mathcal S_{g} d\op{vol}_g$, where the volume form $d\op{vol}_g$ is replaced by the functional $\varphi_h$ defined in \eqref{eq:scalcur-nontracial-vol}. Nevertheless, the change of volume form is already included in the definition of $R_{\slashed D_h}$, that is why $R_{\slashed D_h}$ has the volume factor  $e^{(2-m)h}$ while for the scalar
curvature in
\eqref{eq:scalcur-chofscalcur-commutative} the factor is $e^{2h}$.  

Similar to the previous work in noncommutative tori, we would like to compute the gradient of the Einstein-Hilbert action in two different ways:
\begin{enumerate}
    \item The variation of $V_2(1, D_h^2)$,
    \item  The variation of the local formula $\varphi_0(R_{\slashed D_h})$.
\end{enumerate}
First, we recall the definition of G\^ateaux differential.
\begin{defn}
    \label{defn:scalcur-Gdiff-varation}
    For a fixed $h = h^* \in C^\infty(M_\theta)$, the gradient $\grad_h F_{\op{EH}} \in C^\infty(M_\theta)$  of  the Einstein-Hilbert action with respect to the functional $\varphi_0$ (see \eqref{eq:scalcur-vol-phi0}) is uniquely determined by the following property:
    \begin{align*}
        \varphi_0(a \grad_h F_{\op{EH}} ) = \frac{d}{d\varepsilon}\Big|_{\varepsilon=0}
F_{\op{EH}}(h+a\varepsilon), \,\,\, \forall a=a^* \in C^\infty(M_\theta).
    \end{align*}
\end{defn}

For the variation of the heat kernel, the computation is exactly given in the proof of  in \cite{MR3194491}, we recall the results below,  the proofs for Proposition  \ref{prop:scalcur-varofheatop} and  Corollary \ref{cor:scalcur-KEHandHEH} can be found in Theorem $4.8$ in \cite{MR3194491}. 
\begin{prop}
     \label{prop:scalcur-varofheatop}
    We fix a log-Weyl factor $h$, for any $a \in C^\infty(M_\theta)$, let $$D_\varepsilon \defeq D_{h+a\varepsilon}=e^{(h+\varepsilon a)/2}\slashed D e^{(h+\varepsilon a)/2}.$$ The variation $V_2(1, D^2_\varepsilon)$ is given by:
    \begin{align}
        \label{eq:scalcur-varofheatop}
        \frac{d}{d\varepsilon}\Big|_{\varepsilon=0}   V_2(1, D^2_\varepsilon) = \frac{2-m}{2} V_2( \int_{-1}^1 
        e^{uh/2} a e^{-uh/2} du, D^2_h
        ) .
    \end{align}
\end{prop}

\begin{cor}
    \label{cor:scalcur-KEHandHEH}
   Let $m = \dim M$ and $j_0 = (m-2)/2$. The gradient $\grad_h \op{EH}$ is given by
    \begin{align*}
        \grad_h F_{\op{EH}}  &=  e^{-2j_0h} \brac{ K_{\op{EH}}(\logmodop)(\nabla^2 h) + 
        H_{\op{EH}}(\logmodop,\logmodop)(\nabla h \nabla h)} g^{-1} \\
        &+ (-2j_0)c(m)  e^{-2j_0h}  \mathcal S_\Delta, 
    \end{align*}
    where 
    \begin{align*}
        K_{\op{EH}}(s) &= -8 j_0\frac{ \sinh(s/4)}{s} \mathcal K_{\slashed D_h}(s),\\
        H_{\op{EH}}(s,t) &=-8 j_0\frac{ \sinh((s+t)/4)}{s+t} \mathcal H_{\slashed D_h}(s,t).
    \end{align*}
\end{cor}

Now let us compute the variation from the second perspective. Due to  the trace property of $\varphi_0$:
\begin{align*}
    \varphi_0(\modop(\psi)) = \varphi_0(e^{-2h} \psi  e^{2h}) = \varphi_0(\psi), \,\,\,  \forall 
    \psi \in \Gamma(\End(\slashed S)).
\end{align*}
The local formula in \eqref{eq:scalcur-EH-defn} can be simplified to the one given in Proposition \ref{prop:scalcur-EHlocal}. 

\begin{lem}
    For Schwartz functions $T(s)$ and $H(s,t)$, we have
    \begin{align*}
        \varphi_0\brac{e^{-2j_0h} T(\logmodop)(\nabla^2 h) g^{-1}} = T(0) 
        \varphi_0\brac{e^{-2j_0h} (\nabla^2 h) g^{-1}}
    \end{align*}
    and
    \begin{align*}
        \varphi_0\brac{e^{-2j_0h} H(\logmodop,\logmodop)(\nabla h\nabla h) g^{-1}
    } = \varphi_0\brac{ e^{-2j_0h} G(\logmodop)(\nabla h) (\nabla h) g^{-1}
},
    \end{align*}
with $G(s) = H(s,-s)$.
\end{lem}
As consequences, we get
\begin{align*}
   \varphi_0\brac{
   e^{(2-m)h} \mathcal K_{\slashed D_h}(\logmodop) (\nabla^2 h) g^{-1}
    }&= \mathcal K_{\slashed D_h}(0) \varphi_0(e^{(2-m)h} (\nabla^2 h)g^{-1}), \\
    \varphi_0 \brac{
    e^{(2-m)h}\mathcal H_{\slashed D_h}(\logmodop_{(1)}, \logmodop_{(2)})(\nabla h \nabla h) g^{-1}
    }& =
    \varphi_0 \brac{
    e^{(2-m)h} G(\logmodop)(\nabla h) (\nabla h) g^{-1}
    },
\end{align*}
where $G(s) = \mathcal H_{\slashed D_h}(s,-s)$. 
The two terms above on the right hand side are indeed of the same form by an  
integration by parts argument
\begin{align*}
    \varphi_0(e^{(2-m)h} (\nabla^2 h)g^{-1})& = -\varphi_0\brac{
        (\nabla e^{(2-m)h})(\nabla h) g^{-1}
    }\\ 
    &= -(2-m)  \varphi_0 \brac{ f_1 \brac{(2-m) \logmodop/2}(\nabla h) (\nabla h)g^{-1} }.
\end{align*}

So far, we have proved the following proposition.

\begin{prop}
    \label{prop:scalcur-EHlocal}
    The Einstein-Hilbert action \eqref{eq:scalcur-EH-defn}     
    can be simplified to the following form:
    \begin{align}
        \label{eq:scalcur-EHlocalsimplfied}
        F_{\mathrm{EH}} (h) =   \varphi_0 \brac{
     e^{(2-m)h} T(\logmodop)(\nabla h) (\nabla h) g^{-1}  
 } + \varphi_0\brac{
     e^{(2-m)h} c_{\Delta_k} \mathcal S_\Delta
 },
    \end{align}
    where
    \begin{align*}
        T(s) = -(2-m)\mathcal K_{\slashed D_h}(0) f_1\left (\frac{(2-m)s}{2} \right)
        + \mathcal H_{\slashed D_h}(s,-s),
    \end{align*}
    where $f_1(s)$ is defined in \eqref{eq:scalcur-f1g2}.
\end{prop}
 Based on \eqref{eq:scalcur-EHlocalsimplfied}, the calculation for the variation $\frac{d}{d\varepsilon}|_{\varepsilon = 0} F(h+\varepsilon a)$ is quite similar to the  calculation shown in \cite{MR3369894}, Theorem 5.1. 

 Let $a \in C^\infty(M_\theta)$ be a self-adjoint operator.  Recall $\logmodop = -2\op{ad}_h$ and denote $$\logmodop_{h+\varepsilon a} = -2\op{ad}_{h+\varepsilon a} = \logmodop- 2\varepsilon \op{ad}_a.$$ 

 \begin{prop}
     \label{prop:scalcur-varlocalEH}
     Let $m=\dim M$ and $j_0 = (m-2)/2$ as before.   For any $h=h^* \in C^\infty(M_\theta)$, the gradient of the action $F_{\mathrm{EH}}$ along $h$ $($see \textup{Definition  \ref{defn:scalcur-RDh}}$)$ is equal to
     \begin{align}
      \label{eq:scalcur-varlocalEH}
         \grad_h   F_{\mathrm{EH}} &=  e^{-2j_0h} \brac{\tilde K_{\op{EH}}(\logmodop)(\nabla^2 h) + 
      \tilde  H_{\op{EH}}(\logmodop,\logmodop)(\nabla h \nabla h)} g^{-1} \\
        &+ (-2j_0)c(m)  e^{-2j_0h}  \mathcal S_\Delta, 
     \end{align}
     where
     \begin{align*}
         \begin{split}
         \tilde K_{\op{EH}}(s) & = - (T+\tilde T)(s)  \\
\tilde H_{\op{EH}}(s,t) & = (L+M-P-Q)(s,t),
         \end{split}
     \end{align*}
     where $T$ and $\tilde T$ are defined in \eqref{eq:scalcur-T} and \eqref{eq:scalcur-tildeT} and for $L, M, P$ and $Q$, see \eqref{eq:L-scalcur}, \eqref{eq:M-scalcur}, \eqref{eq:P-scalcur} and \eqref{eq:Q-scalcur}.  
 \end{prop}
 \begin{rem}
     One  discovers the internal relations discussed in Section \ref{sec:scalcur-intrelation} by comparing the local curvature functions in the proposition above and in Corollary \ref{cor:scalcur-KEHandHEH}.
 \end{rem}

 \begin{proof}
   The variation $\frac{d}{d\varepsilon}|_{\varepsilon = 0} F_{\mathrm{EH}}(h+\varepsilon a)$ is decomposed to five parts by the Leibniz rule:
\begin{align*}
    \frac{d}{d\varepsilon}|_{\varepsilon = 0} F_{\mathrm{EH}}(h+\varepsilon a)& =
    \varphi_0 \left( \frac{d}{d\varepsilon}\Big|_{\varepsilon=0} e^{(2-m)(h+\varepsilon a)}
        T(\logmodop)(\nabla h) \nabla h g^{-1}\right)
        \\
        &+  \varphi_0\brac{ e^{(2-m)h} \frac{d}{d\varepsilon}\Big|_{\varepsilon = 0}
       T(\logmodop_{h+\varepsilon a})(\nabla h) (\nabla h) g^{-1} 
        } \\
        &+ \varphi_0\brac{ e^{(2-m)h}
       T(\logmodop)(\nabla a) (\nabla h) g^{-1} 
        } \\
        &+ \varphi_0\brac{ e^{(2-m)h}
       T(\logmodop)(\nabla h) (\nabla a) g^{-1} 
        }\\
        &+ \varphi_0\brac{ \frac{d}{d\varepsilon}\Big|_{\varepsilon = 0} e^{(2-m)(h+\varepsilon a)}  c(m) \mathcal S_\Delta
    }.
\end{align*}
Individually, they can be rewritten in the form:
\begin{align*}
    \varphi_0 \brac{ W_1(\logmodop)(\nabla^2 h) +W_2(\logmodop,\logmodop)(\nabla h \nabla h)}  g^{-1}.
\end{align*}
The computation for the functions $W_1$ and $W_2$ above is distributed in several lemmas listed in  Table \ref{tab:scalcur-lemmas}. 
\begin{table}[!htbp]
    \centering
    \begin{tabular}{|l|l|} \hline
       $ \varphi_0 \left( \frac{d}{d\varepsilon}\Big|_{\varepsilon=0} e^{(2-m)(h+\varepsilon a)}
        T(\logmodop)(\nabla h) \nabla h g^{-1}\right)$ & Lemma \ref{lem:scalcur-localvar-1} \\
$\varphi_0\brac{ e^{(2-m)h} \frac{d}{d\varepsilon}\Big|_{\varepsilon = 0}
       T(\logmodop_{h+\varepsilon a})(\nabla h) (\nabla h) g^{-1} 
        }$ &  Lemma \ref{lem:scalcur-localvar-2}\\
$\varphi_0\brac{ e^{(2-m)h}
       T(\logmodop)(\nabla a) (\nabla h) g^{-1} 
        }$ & Lemma \ref{lem:scalcur-localvar-3}\\
        $\varphi_0\brac{ e^{(2-m)h}
       T(\logmodop)(\nabla h) (\nabla a) g^{-1} 
        }$ & Lemma \ref{lem:scalcur-localvar-4}\\
        $\varphi_0\brac{ \frac{d}{d\varepsilon}\Big|_{\varepsilon = 0} e^{(2-m)(h+\varepsilon a)}  c(m) \mathcal S_\Delta
    }$ & Special case in Lemma \ref{lem:scalcur-localvar-1} \\
        \hline
   \end{tabular}
\caption{ The simplification of  terms appeared in the left column can be found in the lemmas listed in the right column.}
    \label{tab:scalcur-lemmas}
\end{table}

 \end{proof}
 
  The lemmas stated below 
 (from Lemma  \ref{lem:scalcur-localvar-1} to Lemma \ref{lem:scalcur-localvar-4}) are almost identical to those in \cite{MR3369894}, at the end of Section 5.  
\begin{lem}
    \label{lem:scalcur-localvar-1}
    Let $a$ be a self-adjoint element in $C^\infty(M_\theta)$ and $\varepsilon>0$. The variation
    \begin{align*}
        \varphi_0\brac{
        \frac{d}{d\varepsilon}\Big|_{\varepsilon=0} e^{(2-m)(h+\varepsilon a)}
        T(\logmodop)(\psi) \psi
    } =
    \varphi_0\brac{
        a e^{(2-m)h} L(\logmodop,\logmodop)(\psi \psi)
    },
    \end{align*}
    where 
    \begin{align*}
        L(s,t) = T(s) (2-m) f_1((2-m)(s+t)/2) = 2  \frac{e^{(2-m)(s+t)/2}-1}{(s+t)}T(s).
    \end{align*}
    In particular,  when reduced to the commutative setting, we have:
    \begin{align*}
        \varphi_0\brac{ \frac{d}{d\varepsilon}\Big|_{\varepsilon=0} e^{(2-m)(h+\varepsilon a)} \mathcal S_\Delta}
        = (2-m) \varphi_0 \brac{
            e^{(2-m)h} \mathcal S_\Delta
        },  
    \end{align*}
    where $\mathcal S_\Delta$ is the scalar curvature function for the original metric which commutes with everything even after the deformation. 
\end{lem}
\begin{lem}
    \label{lem:scalcur-lemma-nablaT}
    Keep the notations. For any $\psi \in \Gamma(\slashed S)$, set
    \begin{align*}
        \modop(\psi) = e^{-2h} \psi e^{2h}, \,\,\logmodop = -2\op{ad}_h, \,\,\,
        \sigma_s(\psi) =\modop^{-is}(\psi),\,\,\, \text{for $s\in \C$}. 
    \end{align*}
    Then 
    \begin{align}
        \label{eq:scalcur-lemma-nablaT-eq1}
        \nabla \brac{ \sigma_s(\psi)} - \sigma_s \brac{
        \nabla \psi
    } = is \int_0^1 \sigma_{us}\circ \op{ad}_{2\nabla h}\circ \sigma_{(1-u)s} du.
    \end{align}
    Moreover, the obstruction of switching the classical differential and the modular derivation can be seen as follows: 
    \begin{align*}
        \nabla\left(  T(\logmodop)(\psi) \right) =T(\logmodop)(\nabla \psi) 
        + L_1(\logmodop,\logmodop)\left(  (\nabla h) \psi\right)
        + L_2(\logmodop,\logmodop)\left( \psi  (\nabla h) \right),
    \end{align*}
    where 
    \begin{align}
        \label{eq:scalcur-L1andL2}
        L_1(s,t) = (-2)\frac{T(s+t) - T(t)}{s}, \,\,\,
        L_2(s,t) =  2\frac{T(s+t) - T(s)}{t}.
    \end{align}
\end{lem}
\begin{rem}
This is also a straightforward consequence of Proposition 3.15 in  \cite{leschdivideddifference}.
\end{rem}

\begin{lem} \label{lem:scalcur-localvar-2}
    Keep the notations. For any $\psi \in \Gamma(\slashed S)$, we have 
    \begin{align*}
        \varphi_0\brac{e^{(2-m)h}  \frac{d}{d\varepsilon}\Big|_{\varepsilon=0}
        T(\logmodop_{h+\varepsilon a})(\psi) \psi)
    } = 
    \varphi_0\brac{
    a e^{(2-m)h} M(\logmodop,\logmodop)(\psi \psi)
    },
    \end{align*}
    where
    \begin{align*}
        M(s,t) = 2 e^{(2-m)(s+t)/2}
        \left[ \frac{T(-t) - T(s)}{s+t} + \frac{T(t) - T(-s)}{s+t } e^{-(2-m)t/2}\right].
    \end{align*}
\end{lem}

\begin{lem} \label{lem:scalcur-localvar-3}
    Keep the notations.
    \begin{align*}
        &\,\, \varphi_0\brac{e^{(2-m)h} T(\logmodop)(\nabla h) (\nabla a )g^{-1} }\\
        =&\,\, 
        - \varphi_0\brac{a e^{(2-m)h} T(\logmodop)(\nabla^2 h)  } - 
        \varphi_0\brac{a e^{(2-m)h}
    P(\logmodop,\logmodop) (\nabla h \nabla h) g^{-1}
    }
    \end{align*}
    where
    \begin{align*}
        P(s,t) =(2-m) f_1 \brac{ \frac{(2-m)s}{2}} T(t) + L_1(s,t) + L_2(s,t),
    \end{align*}
    with $L_1$ and $L_2$ given in \eqref{eq:scalcur-L1andL2}. 
    \end{lem}
%
    \begin{lem} \label{lem:scalcur-localvar-4}
        Keep the notations.
        \begin{align*}
            &\,\, \varphi_0\brac{e^{(2-m)h} T(\logmodop)(\nabla a) (\nabla h )g^{-1} }\\
        =&\,\, 
        - \varphi_0\brac{a e^{(2-m)h} \tilde T(\logmodop)(\nabla^2 h)  } - 
        \varphi_0\brac{a e^{(2-m)h}
    Q(\logmodop,\logmodop) (\nabla h \nabla h) g^{-1}
    },
                \end{align*}
                where
                \begin{align*}
                    \tilde T(s) &= T(-s) e^{(2-m)s/2}, \\
                    Q(s,t) &= (2-m) f_1 \brac{ \frac{(2-m)s}{2}} \tilde T(t)
                    + (-2)\frac{\tilde T(s+t) -\tilde T(t)}{s}\\
                    &+ 2\frac{\tilde T(s+t) - \tilde T(s)}{t}.
                \end{align*}
    \end{lem}

\section{Widom's Pseudo Differential Calculus}
\label{sec:widomcal}

We will give a quick review of Widom's pseudo differential calculus for vector bundles. Missing proofs can be found in \cite{MR560744,MR538027}.   At the end, we compute the symbol of $\slashed D^2$ in two different ways as an example to demonstrate the difference between the calculus on functions and on spinor sections, also to help the reader get use to the notations. To pass to the noncommutative setting, all the notations stay the same, only the
internal operations between functions and tensor fields are deformed. What we need for the later computations is only the first three terms
\eqref{eq:pcal-a0-a2terms-bundle} in the product formula of two symbols. 

\par

\subsection{Phase functions and local parallel-transport}
\label{subsec:phasefunandlocalpara}

Let $M$ be a smooth manifold with a connection $\nabla$, we define a  phase function $\ell(\xi_x,y) \in C^\infty(T^*M \times M)$, which plays the role of $\abrac{x-y,\xi}$ in classical Fourier analysis on $\R^n$, via the following property:
\begin{enumerate}
    \item For each $x \in M$, $\ell(\cdot,x)$ is a linear function on fibers of $T^*M$.
    \item For each $\xi_x \in T_x^*M$, 
        \begin{align}
            \partial^k_y \ell(\xi_x,y)\Big|_{y=x} = \begin{cases}
                \xi_x, & k=1, \\ 0, & k\neq 1,
            \end{cases}
            \label{eq:pcal-linearfun}
        \end{align}
        where $\partial^k_y \defeq \op{Sym} \circ \nabla^k$ is the symmetrized $k$-th covariant derivative with respect to $y$. 
\end{enumerate}
Such functions $\ell$ are not unique, can be constructed in local charts and pasted together by partition of unity, see \cite{MR560744}. Nevertheless, the covariant derivatives $ \nabla^k_y \ell(\xi_x,y)$ at $y=x$ are totally determined via the property  \eqref{eq:pcal-linearfun}. In practice, there is no harm to treat $\ell(\xi_x,y)$ as $\abrac{\xi_x,\exp_x^{-1} y}$ when performing local computation since only the covariant derivatives at $y=x$ are concerned. Such geometric interpretation was explained in \cite{MR2627820} and \cite{MR973171}. \par

Let $E\rightarrow M$ be a vector bundle equipped with a connection, still denoted by $\nabla$. As before, we can associate a smooth map 
\begin{align*}
    \tau: E \times M \rightarrow E
\end{align*}
to the connection with similar properties
\begin{enumerate}
\item For each $x \in M$, $\tau(\cdot,x):E\rightarrow E$ is a bundle mapping. 
\item When $s_x \in E_x$ is fixed with base point $x \in M$, 
\begin{align}
            \partial^k_y \tau(s_x,y)\Big|_{y=x} = \begin{cases}
                s_x, & k=0, \\ 0, & k\neq 0.
            \end{cases}
            \label{eq:pcal-bunconstsections}
        \end{align}
        Again,  $\partial^k_y$ is the  symmetrized  covariant derivative in $y$.
 \end{enumerate}
 Similar to the phase function $\ell$, the function $\tau$ is not unique and can be constructed locally and then pasted together, see \cite{MR560744}.  The covariant derivatives at $y=x$ are uniquely determined by the equation \eqref{eq:pcal-bunconstsections}. Similar to the phase function, the local interpretation of $\tau(s_x,y)$ (that is, when $x,y$ are closed enough) is nothing but parallel-transport of $s_x$ from $E_x$ to $E_y$ along the unique geodesic joining $x$ and $y$. \par

To get better understanding of the notations, we state the covariant Taylor expansion  for a smooth section $s(y) \in \Gamma(E)$ near by $x \in M$. Recall that $\ell(\xi,y)$ is linear in $\xi \in T^*M$, via the duality, $\ell(\cdot,y)$ can be treated as a vector field on $M$. With this interpretation, 
\begin{align*}
    \ell(\cdot,y)^k = \ell(\cdot,y) \otimes \cdots \otimes \ell(\cdot,y)
\end{align*}
is a symmetric $k$-contravariant tensor field and for a smooth function $f\in C^\infty(M)$, the contraction $\nabla^kf\cdot \ell(x,y)^k$ plays the role of $f^{(k)}(x)(y-x)^k$ in the Taylor's theorem in calculus. If $f(y)\in\Gamma(E)$ is a section of a vector bundle $E$, we need a notion of ``constant sections'': $\tau(\nabla^jf(x),y)$, to replace $f^{(k)}(x)$. We state the covariant Taylor expansion as below. 
\begin{prop}
Let $f(y)\in \Gamma(E)$ be a smooth section, near by a given point $x \in M$, and for a given integer $N$, the following difference
\begin{align*}
    f(y) - \sum_{j=0}^N \frac{1}{j!} \tau(\nabla^jf(x),y) \cdot \ell(x,y)^j
\end{align*}
vanishes to order $N+1$ at $y=x$. Notice that $\nabla^jf(x) \in E_x \otimes_j T_x^*M$, the local parallel-transport $\tau(\nabla^jf(x),y) \in E_y \otimes_j T_x^*M$ acts only on the $E$-bundle part of  $\nabla^jf(x)$, that is $\tau((\cdot)_x,y): E_x \rightarrow E_y$, and leave the $j$-covariant tensor untouched.   
\end{prop}


\subsection{Vertical and Horizontal Differential}
Let $\pi: T^*M \rightarrow M$ be the natural projection and $\mathcal T M$ be the bundle of tensor fields of all rank over $M$. Denote by $\mathcal B M \defeq \pi^* \mathcal TM$ the pull-back tensor bundle.  From analytic point of view, the main difference between the usual tensor fields and the pull-back tensor fields (sections of $\mathcal BM$) is that the base point of the later one depends on both $(x,\xi)$ in local coordinates. \par

\begin{defn}[Horizontal differential]
    Given a pull-back tensor fields $p(\xi_x) \in \Gamma(\mathcal BM)$ and an integer $j\ge 0$, the $j$-th vertical differential $D^j$ is  the usual differential along the fibers $T^*M$ which can be identified with $\R^n$. More precisely, $D^j p$ is a $j$-contravariant tensor, when contracting with $w_1\otimes \dots \otimes w_j$, where $w_1,\dots,w_j\in T^*_xM$,
    \begin{align*}
        & D^jp \cdot (w_1\otimes \dots \otimes w_j)\Big|_{\xi_x}\\
        \defeq&\,\,
        \frac{d}{ds_1}\Big|_{s_1=0}\dots \frac{d}{ds_j} \Big|_{s_j=0}
        p \brac{
        \xi_x+s_1w_1+\dots + s_jw_j
        }. 
    \end{align*}
\end{defn}

The generalization of the vertical differential in a coordinate free way is much less obvious. It involves both $(x,\xi)$ in local coordinates. 

\begin{defn}
    Given a covector $\xi_x \in T^*_xM$, we extend it to a pull-back tensor field near by $x$ via the phase function $d\ell(\xi_x,y)$ (since $d\ell(\xi_x,y)|_{y=x} = \xi_x$). Therefore for any pull-back tensor field $p \in \Gamma(\mathcal BM)$, $p(d\ell(\xi_x,y))$ is a tensor field in $y$ supported near by $x$. We define the $j$-th horizontal differential to be
    \begin{align*}
        \nabla^j_y p(d\ell(\xi_x,y))\Big|_{y=x}.
    \end{align*}
\end{defn}

The definition above extends quickly to the case in which $p \in \Gamma(\mathcal BM\otimes \pi^*\mathcal E)$ where $\mathcal E \rightarrow M$ is a vector bundle with a connection. What we need later is that $\mathcal E = \Hom(E,F)$, with a induced connection from connections on $E$ and $F$, where $E,F\rightarrow M$ are the domain and the range respectively of pseudo differential operators. 

\begin{eg}
    \label{eg:pcal-symofDelta}
 As an example, we show that for the squared length function $\abs\xi^2 \in C^\infty(T^*M)$ and the Levi-Civita connection, $\nabla \abs\xi^2 =0$. Indeed, set $p(\xi) = \abs\xi^2$, we write $p(d\ell(\xi_x,y))$ in terms of  tensor notation:
\begin{align*}
    p(d\ell(\xi_x,y)) =( d\ell(\xi_x,y)\otimes d\ell(\xi_x,y)) \cdot g^{-1},
\end{align*}
where $g^{-1}$ is the metric tensor on $T^*M$. According to the Leibniz property and the fact that $\nabla g^{-1} =0$,
\begin{align*}
    \nabla  p(d\ell(\xi_x,y)) 
    &= ( \nabla d\ell(\xi_x,y)\otimes d\ell(\xi_x,y)) \cdot g^{-1} 
    \\
    &+ ( d\ell(\xi_x,y)\otimes \nabla d\ell(\xi_x,y)) \cdot g^{-1}.
\end{align*}
Since $\nabla$ is torsion free, $\nabla d\ell(\xi_x,y) = \nabla^2 \ell(\xi_x,y) = \symd^2 \ell(\xi_x,y)$, which becomes zero when evaluating at $y=x$. It is worth  pointing out that $\nabla^2 \abs\xi^2\neq 0$, since $\nabla  p(d\ell(\xi_x,y))$ defined above only vanishes at $y=x$. 
   
\end{eg}

\subsection{Symbol and Quantization maps}
Let us consider $\Psi(E,F)$, pseudo differential operators from bundle $E$ to $F$. We fix a phase function $\ell$ and choose a cut-off function $\alpha(x,y) \in C^\infty(M\times M)$, which equal to $1$ in a small neighborhood of the diagonal with the property that the support is so closed to the diagonal that whenever $\alpha(x,y)\neq 0$, $d\ell(\xi_x,y)$ is nonzero for all $\xi_x \neq 0$. 

For a pseudo
differential operator $P: \Gamma(E) \rightarrow \Gamma(F)$, its symbol $\sigma_P$ belongs to  $$\Gamma(T^*M, \pi^* \Hom(E,F))$$ defined by the following formula: for $\xi_x \in T^*_xM$ and $s_x \in E_x$
\begin{align}
    \sigma_P(\xi_x)(s_x) \defeq P_y \brac{
        \alpha(x,y) \tau(s_x,y) e^{i\ell(\xi_x,y)}
},
    \label{eq:pcal-symmap}
\end{align}
the subscript $y$ indicates that $P$ acts on the $y$ variable. \par
For the quantization map $\op{Op}$ from symbols to pseudo differential operators, the construction is not unique. The explicit formula of $\op{op}$ will not be involved in the later computation, what we need is the property that $\op{Op}$ and $\sigma$ are inverse to each other modulo smoothing operators (symbols).\par

For instance, we can quantize the symbols $p(\xi_x):T^*M\rightarrow \End(E)$ in a geometric way: for any section $f \in \Gamma(E)$, near by a point $x\in M$, we lift it to a function on $T_xM$: $f_x: T_xM \rightarrow E$ with $f_x(Y) =\alpha(x,y) \tau(f(y),x)$, where
$y=\exp_xY$ and $\alpha(x,y)$ is the cut-off function defined before.  

\begin{align*}
    \op{Op}(p)f(x) = (2\pi)^{-\dim M} \int_{T_xM \times T_x^*M}
    e^{-i\abrac{\xi_x,Y}} p(\xi_x) f_x(Y) dYd\xi_x.
\end{align*}
Different choices of the cut-off function $\alpha(x,y)$ give rise to different quantization maps but they all agree with each other modulo smoothing operators. 

\subsection{Trace formula of symbols}

The symbol calculus make sense only for the complete symbols, namely everything is up to a smoothing operator. Therefore there must exit some errors when we try to compute the trace of the operator via its symbol.  To measure the errors, we introduce a parameter $t \in [1,\infty)$, for a pseudo differential operator $P$ with symbol $p(x,\xi)$, $P_t$ denotes the family of operators with symbol $p_t(x,\xi) \defeq p(x,t\xi)$. 
    \par  
Let $(M,g)$ be a $m$ dimensional closed Riemannian manifold. The canonical $2m$-form on the cotangent bundle $T^*M$ is given in local coordinates $(x,\xi)$ by
\begin{align}
\label{eq:wpse-diff-canonical-twom-form}
	\Omega = dx_1 \wedge \cdots \wedge dx_m \wedge d\xi_1 \wedge \cdots \wedge d\xi_m,
\end{align}
while $\Omega = dg d\xi_{x,g^{-1}}$ with the volume form $dg = (\det g) dx_1 \wedge \cdots dx_m$ and 
\begin{align}
\label{eq:wpse-diff-measure-on-T^*_xM}
	d\xi_{x,g^{-1}} = (\det g)^{-1}d\xi_1 \wedge \cdots \wedge d\xi_m
\end{align}
defines a measure on the fiber $T^*_xM$. 

\begin{prop}\label{prop:wpse-diff-sym-to-kernel-dia}
Keep the notations as above, let $P$ be pseudo differential operator acting on spinor sections $\Gamma(\slashed S)$ with order  less than $m$, the dimension of the manifolds such that the Schwartz kernel
function $k_P(x,y)$ exists, then the Schwartz kernel of the dilation  $P_t$   on the diagonal is given by 
    \begin{align}
        k_{P_t}(x,x) = \frac{t^m}{(2\pi)^m} \int_{T^*_xM}  \sigma(P) d\xi_{g^{-1}} + O(t^{-N}), \,\,\, \forall N \in\N.
        \label{eq:wpse-diff-sym-to-kernel-dia}
    \end{align}
 In particular, we obtain the trace formula for $P_t$:
 \begin{align}
     \Tr(P_t) = \frac{t^m}{(2\pi)^m} \int_{T^*M} \Tr_x \sigma(P) \Omega + O(t^{-N}), \,\,\, \forall N \in\N,
 \end{align}   
 where $\Tr_x$ is the fiberwise trace and   $\Omega = dg d\xi_{x,g^{-1}}$ is the canonical $2m$-form defined in \eqref{eq:wpse-diff-canonical-twom-form}. 
\end{prop}

 See \cite[Theorem 5.7]{MR538027} for the proof.

\subsection{Product formula for symbols}
Let $\ell$ and $\tau$ be a phase function and a local parallel transport defined in section \ref{subsec:phasefunandlocalpara}. Recall that for $j\in \Z_+$, the tensor fields $\nabla_y^j \tau(s_x,y)$ and $\nabla_y^j \ell(\xi_x,y)$ when evaluated at $y=x$ is independent of the choice of $\tau$ and $\ell$. They are simply denoted as $\nabla^j \ell$ and $\nabla^j \tau$ in the rest of the paper. 

\begin{prop} \label{prop:pcal-symproduct}
    Let $P,Q: \Gamma(E) \rightarrow \Gamma(E)$ be two pseudo differential operators.  Then the symbol of the composition $PQ$ is given by 
    \begin{align}
        \sigma_{PQ} &= \sigma_{P} \star \sigma_{Q} 
       \backsim \sum \frac{i^{k -\bar p -\sum_{0}^{k}(m_i+p_i)}}{k! m_0! \cdots m_k! p_0! \cdots p_k! \bar p!} \nonumber\\
      & \times (D^{\bar p+ \sum_{0}^{k} p_i}\sigma_P) ( \nabla^{\bar p} D^{\sum_{0}^{k} m_i}\sigma_Q)
       \nabla^{p_0+m_0} \tau \nabla^{p_1+m_1} \ell \cdots \nabla^{p_k+m_k}\ell, 
        \label{eq:pcal-symproduct}
    \end{align}
    where the sum runs over all
    \begin{align*}
    k\ge 0, \,\,\, m_1,\dots, m_k \ge 2,\,\,\, \bar p, p_0, \dots, p_k \ge 0. 
    \end{align*}
\end{prop}
Some words about the notations.
\begin{rem} \hspace{1cm}
    \begin{enumerate}
        \item In \eqref{eq:pcal-symproduct}, the contraction of contravariant and covariant tensors produces a scalar function. At each point, $\sigma_P \circ \sigma_Q \circ \tau$ are compositions of endomorphisms on the bundle $E$. Therefore the summands in  \eqref{eq:pcal-symproduct} are $\End(E)$-valued functions on $T^*M$ as it should be. 
        \item The contravariant tensors $\nabla^{p_0+m_0} \tau$, $\nabla^{p_1+m_1} \ell$, \dots, are not symmetric, therefore the order of the summation $p_j+m_j$ matters and indicates that the $m$ indices act first. Since $D^j$ are symmetric tensors for any power $j$, we can replace $\nabla^{p_s+m_s} \ell$ by $\partial^{p_s}\partial^{m_s}\ell$, which stands for symmetrizing the $m_s$ indices and $p_s$ indices separately. 
         For example, take $\sigma_P = \sigma_Q=\sigma_{\slashed D}$ are both equal to the symbol of the Dirac operator, consider a term $D\sigma_{\slashed D} D \sigma_{\slashed D} \nabla^{2}\tau$, then the contraction is given by
            \begin{align*}
                \sum_{j,k}   (D \sigma_{\slashed D})_k (D \sigma_{\slashed D})_j (\nabla^2 \tau)_{jk},
            \end{align*}
            where the $j$ index acts first on $\tau$.
    \end{enumerate}
\end{rem}
\begin{prop}
    \label{prop:pcal-symprod-explicit-terms}
    After excluding the zero terms, the symbol product \eqref{eq:pcal-symproduct} is summed over 
    \begin{align}
    \label{eq:1-pcal-sumover}
    \begin{split}
& k\ge 0,\,\, \bar p\ge 0,\,\,\, p_0\ge 0, \,\,\, m_0\ge 0, \,\,\, p_1, \dots, p_k \ge 1, \,\,\,
    m_1, \dots, m_k \ge 2 \\
   & m_0+ p_0 \neq 1. 
    \end{split}
\end{align}
Furthermore, the symbol product can be grouped in the following way:
\begin{align*}
    \sigma_P \star \sigma_Q\backsim \sum_{j=0}^\infty a_j(\sigma_P,\sigma_Q),
\end{align*}
where each $a_j$ is a bi-differential operator that reduces the total degree by $j$:
\begin{align*}
    a_j(\cdot, \cdot): S\Sigma^d \times S\Sigma^{d'} \rightarrow S\Sigma^{d'+d - j}, \,\,\,\,
    \forall d,d' \in \Z.
\end{align*}
In fact, $a_j$ consists of terms in \eqref{eq:pcal-symproduct} whose power of the complex number $i$ is equal to $j$, namely, $a_j$ is obtained by summing over all the indices such that  
\begin{align*}
    j= \sum_0^k(m_l+p_l) + \bar p -k.
\end{align*}
\end{prop}

First few $a_j$ are:
\begin{align}
 \label{eq:pcal-a0-a2terms-bundle}
    \begin{split}
    a_0(p,q) &= pq, \\
a_1(p,q) &= -i (Dp) (\nabla q), \\
a_2(p,q) &=-\brac{
\frac12 (D^2p) (\nabla^2 q) + (Dp) (Dq) (\nabla^2 \tau)+ \frac12 (Dp) (D^2q) (\nabla^3 \ell )
}.
    \end{split}
\end{align}
\begin{rem}
    \hspace{1cm} 
    \begin{enumerate}
        \item The deformed symbol calculus is simply replace the internal operations by their $\theta$-version. 
        \item Let $\abs\xi^2 \in C^\infty(T^*M)$ be the squared length function. Since the connection is Levi-Civita, the tensor fields $\nabla^j \abs\xi^2$ with $j=0,1,2$,  $(\nabla^2 \tau)$  and $(\nabla^3 \ell )$ are all invariant under isometries, in particular, they are $\T^n$-invariant. They will be the central elements in  the later computation. 
    \end{enumerate}
     
\end{rem}

\subsection{Symbols of $\slashed D$ and $\slashed D^2$}
Let $\slashed S$ be the spinor bundle. For any $\xi_x\in T^*M$, $\mathbf c(\xi_x)$ stands for the Clifford multiplication. 
\begin{prop}
   The symbol of the Dirac operator $\slashed D$ is the tautological section  of the Clifford multiplication in $\Gamma(T^*M, \End(\slashed S))$ upto a factor $\sqrt{-1}$, namely,
   \begin{align}
       \sigma_{\slashed D}(\xi_x) = i \mathbf c(\xi_x).
       \label{eq:pcal-symofD}
   \end{align}
\end{prop}
\begin{proof}
   Since $\slashed D$ is a differential operator, we can ignore the cut-off function $\alpha$ in \eqref{eq:pcal-symmap}. In a orthonormal coordinate system with a basis $\set{e_j}$, for any section $s \in \Gamma( \slashed S)$, the evaluation of $\sigma_{\slashed D}(\xi_x)(s)$ at $x \in M$ is given by:
    \begin{align*}
        \brac{ \sigma_{\slashed D}(\xi_x)(s)}(x) &= \slashed D\brac{
            e^{i\ell(\xi_x,y)} c(s_x, y)
        }\Big|_{y=x} 
        = \sum_j \mathbf(e_j) \nabla_j \brac{
        e^{i\ell(\xi_x,y)} c(s_x, y)
        }\\
        &= \brac{\sum_j i \nabla_j \ell (\xi_x,y) e^{i\ell(\xi_x,y)} \mathbf c(e_j)\tau(s_x,y)
        + e^{i\ell(\xi_x,y)} \nabla_j \tau(s_x,y)} \Big|_{y=x}
        \\
        &= i \mathbf c(\xi_x)(s|_x).
    \end{align*}
    In the last step, we have used the fact that at $y=x$:
    \begin{align*}
        \nabla_j \ell (\xi_x,y) = (\xi_x)_j, \,\,\, \tau(s_x,y) = s_x , \,\,\,
        \nabla_j \tau(s_x,y) = 0.
    \end{align*}

\end{proof}

We now calculate the symbol of $\slashed D^2$ in two ways. First, use the symbol calculus in \eqref{eq:pcal-a0-a2terms-bundle}. Since $\sigma_{\slashed D}(\xi)$ is linear in $\xi$, $D^j \sigma_{\slashed D} =0$ for all $j\ge 2$, in particular, $a_j(\sigma_{\slashed D},\sigma_{\slashed D}) =0$ for $j>2$. Therefore 
\begin{align*}
    \sigma_{\slashed D^2} = \sigma_{\slashed D} \star \sigma_{\slashed D} =
    \sum_0^2 a_j(\sigma_{\slashed D},\sigma_{\slashed D}).
\end{align*}
Start with $a_0$:
\begin{align*}
    a_0(\sigma_{\slashed D},\sigma_{\slashed D})(\xi_x) = \sigma_{\slashed D}(\xi_x)\sigma_{\slashed D}(\xi_x)
    = (i)^2 \mathbf c(\xi_x) \mathbf c(\xi_x) = \abs\xi^2 I,
\end{align*}
here we have used the Clifford relation $\mathbf c(\xi)^2 = - \abrac{\xi,\xi}$. 
For a Clifford connection $\nabla$ on $\mathcal S$ (cf. \cite[Def. 3.39]{berline1992heat}), we alway have 
\begin{align*}
    \nabla_X ( \mathbf c (w)(s)) = \mathbf c(\nabla_X w)(s) + \mathbf c(w) (\nabla_X s),
\end{align*}
where $s\in \Gamma(\slashed S)$ is a spinor section, $X$ is a vector field and $w$ is differential form so that $\mathbf c (w)$ becomes a Clifford algebra section. In particular, $\nabla (\mathbf c (w)) =\mathbf c (\nabla w) $. We are ready to compute the first horizontal differential of $\sigma_{\slashed D}$:
\begin{align*}
    \nabla \sigma_{\slashed D}(\xi_x) = \nabla \brac{
    \sigma_{\slashed D}(d\ell(\xi_x,y)
}\Big|_{y=x} = \nabla \brac{
\mathbf c(d\ell(\xi_x,y))
}\Big|_{y=x} = \mathbf c \brac{
\nabla d\ell(\xi_x,y)
}\Big|_{y=x},
\end{align*}
when evaluating at $y=x$, $\nabla d\ell(\xi_x,x) = \nabla^2 \ell(\xi_x,x) = \symd^2 \ell(\xi_x,x) = 0$, here we have used the torsion free property of $\nabla$, so that $\nabla^2 \ell$ is symmetric. We have proved that $\nabla \sigma_{\slashed D} = 0$, thus 
$a_1(\sigma_{\slashed D},\sigma_{\slashed D}) =0$ as well. 

\par
In $a_2$, the only non-zero term is $-(D \sigma_{\slashed D}) (D\sigma_{\slashed D})(\nabla^2 \tau)$. We first compute $\nabla^2 \tau$ in local coordinates. From the definition of curvature tensor, we have when $y$ is in a small neighborhood of $x$:
\begin{align*}
    (\nabla^2)_{ij} \tau (s_x,y) - (\nabla^2)_{ji} \tau(s_x,y) = K_{ji}(\tau(s_x,y)),
\end{align*}
at $y=x$, by definition we get $\symd^2 \tau(s_x,x) =0$, that is $\nabla^2 \tau$ is anti-symmetric: $(\nabla^2)_{ij} \tau (s_x,x)=- (\nabla^2)_{ji} \tau (s_x,x)$. Therefore:
\begin{align*}
    2 (\nabla^2)_{ij} \tau(s_x,x) = K_{ji}(s_x).
\end{align*}
%
%
%
 Let $\set{e_j}$ be a orthonormal frame at $T^*_xM$, denote $\mathbf c^k = \mathbf c(e_k)$, 
\begin{align*}
    (D \sigma_{\slashed D}) (D\sigma_{\slashed D})(\nabla^2 \tau) &= 
    \sum_{j,k} (D \sigma_{\slashed D})_k (D\sigma_{\slashed D})_j(\nabla^2 \tau)_{jk
    } \\
    &=
   (i)^{2} \sum_{k,j}\frac12 \mathbf c^k \mathbf c^j K(e_k,e_j).
\end{align*}
On the spinor bundle $\slashed S$, the Clifford contraction of the curvature two form $K$: $$\sum_{k,j}\frac12 \mathbf c^k \mathbf c^j K(e_k,e_j)$$ has only the Riemannian part which is equal to $\frac14 S_\Delta$, where $S_\Delta$ is the scalar curvature function. The calculation can be found, for instance, in \cite[Prop. 3.18]{roe1999elliptic}.  Thus 
\begin{align*}
    (D \sigma_{\slashed D}) (D\sigma_{\slashed D})(\nabla^2 \tau) = - \frac14 S_\Delta, 
\end{align*}
which is equal to $-a_2(\sigma_{\slashed D}, \sigma_{\slashed D})$.
\par
Sum up the calculation above, 
\begin{align*}
    \sigma_{\slashed D^2}(\xi_x)& = a_0(\sigma_{\slashed D}, \sigma_{\slashed D})
    + a_2(\sigma_{\slashed D}, \sigma_{\slashed D}) \\
   & = (\abs\xi^2 + \frac14 S_\Delta)I.  
\end{align*}
For the second method, we make use of the Lichnerowicz formula (cf. \cite[Thm 3.52]{berline1992heat}),
\begin{align*}
    \slashed D^2 = \Delta + \frac14 S_\Delta,
\end{align*}
where $\Delta$ is the connection Laplacian. Similar computation as in example \eqref{eg:pcal-symofDelta} shows that the symbol of the connection Laplacian $\sigma_\Delta = \abs\xi^2$. Therefore we see again, that  the symbol of $\slashed D^2$ equals  
\begin{align*}
    \sigma_\Delta + \frac14 \sigma_{S_\Delta} = \abs\xi^2 + \frac14 S_\Delta. 
\end{align*}
%
%

\section{Symbolic Computation}
\label{sec:computation}

 In  the previous work \cite{Liu:2015aa},  the resolvent approximation and tensor computation were discussed in great detail.  Our main goal of this section is the unfinished job in \cite{Liu:2015aa}: to give a detailed exploration of 
 the dependence of local curvature functions on the dimension. 
 \par
%
We start with the Laplacian type operator in \eqref{eq:scacur-defn-Deltak}:
\begin{align*}
    \Delta_k = k^{1/2} (\slashed D_k)^2 k^{-1/2} =k^2 \slashed D^2 + k \mathbf c(\nabla k) \slashed D,\,\,\, k = e^h.
\end{align*}

Recall that the symbol of $\slashed D$ and $\slashed D^2$ are given by:
\begin{align*}
    \sigma(\slashed D)(\xi) = -i \mathbf c(\xi), \,\,\,\,\,
    \sigma(\slashed D^2)(\xi) =\brac{ \abs\xi^2 + \frac14 S_\Delta } \cdot I,\,\,\,\xi \in T^*M.
\end{align*}
Let $\lambda$ be the resolvent parameter, the symbol of $\Delta_k - \lambda$ is of the form $\sigma(\Delta_k - \lambda) = p_2(\xi,\lambda) + p_1(\xi) +p_0(\xi)$,  with 
\begin{align*}
    p_2(\xi,\lambda) &=  (k^2\abs\xi^2 - \lambda) I, \,\,\,\,\\
    p_1(\xi) &=  -i k(\nabla k)  (D\sigma_{\slashed D}) \sigma_{\slashed D},\\
    p_0(\xi) & =( \frac14 k^2 S_\Delta)I.
\end{align*}
Recall the standard resolvent approximation procedure (see for instance, Gilkey \cite{gilkey1995invariance}). We would like to construct a sequence of symbols $\set{b_j}_{j=0}^\infty$ in which $b_j$ is of degree $-j-2$ by solving the equation 
\begin{align*}
    (b_0+b_1 +b_2+\cdots) \star (p_2+p_1+p_0) \backsim 1
\end{align*}
inductively. The first approximation  $b_0 =p_2^{-1}(\xi,\lambda) =((k^2\abs\xi^2 - \lambda))^{-1}$ is the inverse of the leading symbol $p_2$ in the deformed  algebra of complete symbols\footnote{Complete symbols simply means modular smoothing symbols.} $S\Sigma(M_\theta)$. The rest of $b_j$ are given by a recurrence relation, in particular 
\begin{align}
    \label{eq:pcal-b1-sym}
  b_1= \brac{
    b_0p_1+ a_1(b_0,p_2)
    }(-b_0).  
\end{align}
and
\begin{align}
    \begin{split}
    b_2&=\left[\mathit{a}_0\left(b_0,p_0\right)+\mathit{a}_0\left(b_1,p_1\right)+\mathit{a}_1\left(b_0,p_1\right) \right. \\
    &+ \left. \mathit{a}_1\left(b_1,p_2\right)+\mathit{a}_2\left(b_0,p_2\right)\right] (-b_0).
    \end{split}
        \label{eq:pcal-b2-sym}
\end{align}
 We start with \eqref{eq:pcal-b1-sym}:
\begin{align}
    b_1 = -i b_0^2 k^2 (\nabla k^2) b_0 (D\abs\xi^2) \abs\xi^2  + 
    i b_0 k(\nabla k) b_0  (D\sigma_{\slashed D}) \sigma_{\slashed D},
\end{align}
since $\nabla k^2 = (\nabla k)k + k(\nabla k)$,
\begin{align}
    \begin{split}
      b_1=  &  i  k b_0 (\nabla k) b_0 (D\sigma_{\slashed D}) \sigma_{\slashed D}
        -i    k^3 b_0^2 (\nabla k) b_0 \left(D \xi ^2\right) \xi^2\\
   & -i   k^2 b_0^2 (\nabla k) k b_0 \xi ^2 (D\xi ^2).
    \end{split}
  \label{eq:pcal-b1}  
\end{align}

The scalar curvature  $R_{\Delta_k}$ in \eqref{eq:scalcur-R-Deltak} is obtained by the following integration against the next term $b_2(\xi,\lambda)$ (cf. \cite{Liu:2015aa}, Theorem 6.2):
\begin{align}
    \label{eq:modcur-RDlatak-int}
    R_{\Delta_k}(x) = \frac{1}{(2\pi)^m}\int_{T^*_xM} \frac{1}{2\pi i} \int_{C}  e^{-\lambda} b_2(\xi,\lambda) d\lambda dg^{-1}_x,
    \,\,
    m=\dim M,
\end{align}
where $C$ is a contour that defines the heat operator and $dg^{-1}_x$ is the measure associated to the Riemannian metric $g^{-1}|_x$ at $T^*_xM$, in a local chart $(x,\xi)$, $dg^{-1} = \det(g^{-1})d\xi_1 \cdots d\xi_m$. The trace formula in Proposition \ref{prop:wpse-diff-sym-to-kernel-dia} plays a crucial role in the proof of the result. 
\par
The computation for  $b_2$  can be carried out in similar way and was explained in \cite{Liu:2015aa}, Section 7 in great detail. In this paper, we leave the tedious calculation to \textsf{Mathematica}\footnote{The calculation can be handled completely by hands. Nevertheless the assistant of a computer algebra system (CAS) helps to eliminate potential algebraic mistakes and has great significance for deeper exploration in this direction.}. We collect the final results as screenshots and are put them  in appendix \ref{app:completeb2terrm}.  The list of complete terms of $b_2$ contains approximately $30$ terms (see the left column of the screenshots in appendix
\ref{app:completeb2terrm}), which is divided to several parts according to the summands  appeared in \eqref{eq:pcal-b2-sym}.

\subsection{The contour integral and  integration over the cosphere bundle}
On each fiber $T^*_xM$, we compute the integral using spherical coordinates:
\begin{align*}
    \int_{T^*_xM} b_2(\xi,\lambda) d\xi = \int_0^\infty b_2(r,\lambda) (r^{m-1})dr, \,\,\, m = \dim M,
\end{align*}
with
\begin{align}
    \label{eq:modcur-b2rlambda}
    b_2(r,\lambda) = \int_{\mathbb S^{m-1}} b_2(\xi,\lambda) d \sigma_{\mathbb S^{m-1}}, \,\,\, r = \abs\xi,
\end{align}
where $d \sigma_{\mathbb S^{m-1}}$ is the volume form of the unit sphere defined by the original Riemannian metric $g^{-1}$. The result of $b_2(r,\lambda)$, upto an overall constant  $\mathrm{vol}(\mathbb S^{m-1})$, $m=\dim M$, is listed in Table \eqref{tab:intsph}. For terms involving the spinor bundle, the computation can be found in  Appendix \ref{app:intsphere}. For the rest, see \cite{Liu:2015aa}, Appendix C.  
\begin{table}
    \centering
    \begin{tabular}[!htpb]{|l|l|l|l|}
        \hline
      $b_2(\xi,\lambda)$   & $b_2(r,\lambda)$  & $b_2(\xi,\lambda)$ & $b_2(r,\lambda)$ \\ \hline
      $  D^2\abs\xi^2 \nabla^2(\mathit{c})$   & 0 &
      $D^2\abs\xi^2 \nabla^2\abs\xi^2$  & $(4/(3m)) S_\Delta \abs\xi^2$\\
      $(D \abs\xi)^2 \nabla\abs\xi^2$ & 0 &
      $(D\abs\xi) (D^2\abs\xi^2) (\nabla^3 \ell)$ & $(-8/(3m)) S_\Delta \abs\xi^2$ \\
      $(D\sigma_{\slashed D})^2$ & 1 & 
      $(D \abs\xi^2) (D\sigma_{\slashed D}) \sigma_{\slashed D}$ 
      & $(2/m ) \abs\xi^2$ \\
      $D^2\abs\xi^2$ & 2 &
      $((D \sigma_{\slashed D})\sigma_{\slashed D})^2$ & $(2-m)\abs\xi^2/m$ \\ 
      \hline
          \end{tabular}
    \caption{Integration over the unit cosphere bundle. The result $b_2(r,\lambda)$ appeared in the right columns (cf. Eq. \eqref{eq:modcur-b2rlambda}) is upto a scalar factor $\mathrm{vol}(\mathbb S^{m-1})$. Here $S_\Delta$ is the scalar curvature function and $m=\dim M$.}
    \label{tab:intsph}
\end{table}

Sum up, all the terms of $b_2(\xi,\lambda)$ are listed in the left columns  from Fig. \eqref{fig:a_2(b_0,p_2)_partI} to Fig. \eqref{fig:a_0(b_0,p_0)} while the middle columns consists of terms from $b_2(r ,\lambda)$. The  middle columns are obtained by applying the substitutions defined in Table \eqref{tab:intsph}.


As a function on the cotangent bundle, $b_2(\xi,\lambda)$ has the following homogeneity property:
\begin{align}
    \label{eq:modcur-b2homoge}
    b_2(c\xi, c^2\lambda) = c^{-4} b_2(\xi,\lambda), \,\,\, \forall c>0.
\end{align}
Therefore integrating $b_2(\xi,\lambda)$ in $\xi$ makes sense only when $\dim M < 4$. As a consequence, after switching the order of integration in \eqref{eq:modcur-RDlatak-int} by an integration by parts argument, the contour integral in $\lambda$ can be achieved by evaluating at  $\lambda = -1$. We recall Lemma 7.2 in \cite{Liu:2015aa}. 
\begin{lem}
    \label{lem:modcur-put-lamda-to-1}
    Keep the notations. Let $m = \dim M$ and $j_0 = (m-2)/2$, which is the smallest integer so that the homogeneity $($defined in \eqref{eq:modcur-b2homoge}$)$  of $\frac{d^{j_0}}{d\lambda^{j_0}} b_2(r,\lambda)$ is strictly less than $-m$.  Then we have
\begin{align}
 \frac{1}{2\pi i}\int_C e^{-\lambda}  (\int_{0}^{\infty} b_2(r,\lambda)r^{m-1}dr) d\lambda  
 = \int_{0}^{\infty}  \frac{d^j_{0}}{d\lambda^j_{0}}\Big|_{\lambda =-1} b_2(r,\lambda) r^{m-1}dr.
    	\label{eq:modcur-put-lamda-to-1}
\end{align}
    
\end{lem}

\subsection{Rearrangement lemma}
\label{sec:rearrangementlema}
After integrating over the cosphere bundle, the $b_2$ term consists of two types of summands:
\begin{align}
    k^2 f_0(rk^2) \rho f_1(rk^2) \,\,\text{or} \,\, k^2 f_0(rk^2) \rho_1 f_1(rk^2)\rho_2 f_2(rk^2)
    \label{eq:pcal-types-b2term}
\end{align}
where $r\in[0,\infty)$, $k$ is the conformal factor and $f_j$ are smooth functions on $\R_+$ and $\rho_j$ in between are tensor fields over $M$ on which $C^\infty(M_\theta)$ acts from both sides. Recall the modular operator
\begin{align*}
    \modop(\rho) = k^{-2} \rho k^2,
\end{align*}
then the rearrangement lemma states that:
\begin{align*}
    \int_0^\infty k^2 f_0(rk^2) \rho f_1(rk^2) dr &= K(\modop)(\rho), \\
    \int_0^\infty k^2 f_0(rk^2) \rho_1 f_1(rk^2)\rho_2 f_2(rk^2)
    &= H(\modop_{(1)}, \modop_{(2)}) (\rho_1 \cdot \rho_2),
\end{align*}
 the subscript  $j=1,2$ in  $\modop_{(j)}$ indicates that $\modop$ acts on the $j$-th factor of the product $\rho_1 \cdot \rho_2$. The functions $K$ and $H$ are given by
 \begin{align}
     \label{eq:pcal-rearr-funKandH}
     \begin{split}
     K(s) &= \int_0^\infty f_0(r) f_1(sr) dr, \,\, s>0,
\\
     H(s,t) &= \int_0^\infty f_0(r) f_1(sr) f_2(str)dr, \,\, s,t>0.
     \end{split}
  \end{align}
For the proof of the lemma, see \cite{leschdivideddifference} (whose  notations are consistent with ours), also \cite{MR2907006} and  \cite{MR3194491}.

Later, we will  need  the explicit family of functions given below. Let $m=\dim M \ge 4$ be even,
  denote:
\begin{align}
    \label{eq:pcal-contcp}
    c_{(p,m)} &= (p-2+(m-2)/2)!/(p-1)!,\\
    \label{eq:pcal-onevarKpq}
    K_{(p,q,m)}(s) &= \frac{d^{m/2-2}}{du^{m/2-2}}\Big|_{u=0}
    (1-u)^{-p} (s-u)^{-q}, \\
    H_{(p,q,l,m)}(s,t) &= \frac{d^{m/2-2}}{du^{m/2-2}}\Big|_{u=0}
    (1-u)^{-p} (s-u)^{-q} (st-u)^{-l},
    \label{eq:pcal-twovarHpql}
\end{align}
where $s,t \ge 0$ and $p,q,r \in \Z_+$. We remark that the constant $c_{(p,m)}$ comes from the degenerate situation of $K_{(p,q,m)}$  (or $H_{(p,q,l,m)}$) by setting $s=1$ (or $s=t=1$). More precisely, $c_{(p,m)} = K_{(\alpha,\beta,m)}(1) = H_{(\alpha',\beta',\gamma',m)}(1,1)$, where $p=\alpha+\beta = \alpha'+\beta'+\gamma'$.

\subsection{Computation for the local curvature functions}

\label{subsec:comformodcur}

In this section, we will explain using specific examples how to obtain the right column of table \eqref{tab:typesofb2withmodfun} from the left. 

\begin{prop}
    Keep the notations.  All terms of $b_2(r,\lambda)$ belong to one of the types appeared in the left column of table \eqref{tab:typesofb2withmodfun}, while the right column is the final contribution from that term to the  modular curvature in \eqref{eq:scalcur-R-Deltak}. 
\end{prop}

\begin{table}[!htb]
    \centering
    \begin{tabular}{l|l} 
        Types & Modular curvature functions  \\ \hline
     $\abs\xi^{\nu_1} k^{\mu_1} b_0^p (\nabla^2k) b_0^q$ &
     $k^{1-m} K_{(p,q,m)}(\modop)(\nabla^2 k)$ \\
     $\abs\xi^{\nu_2} k^{\mu_2} b_0^p (\nabla k) b_0^q (\nabla k) b_0^l$ &
     $k^{-m}H_{(p,q,l,m)}(\modop, \modop)(\nabla k \nabla k)$  \\ 
     $\abs\xi^{\tilde \nu_2} k^{\tilde \mu_2} b_0^p (\nabla k)^2 b_0^q$ &
     $k^{-m}K_{(p,q,m)}(\modop)(\nabla k \nabla k)$ \\
     $\abs\xi^{\nu_0} k^{\mu_0} b_0^p S_\Delta$ & 
     $k^{2-m}c_{(p,m)}S_\Delta$
    \end{tabular}
    \caption{Types of terms in $b_2(r,\lambda)$ and the resulting modular curvature function, there is an overall factor $1/2$ which is suppressed.Functions appeared in the right column are defined in \eqref{eq:pcal-contcp}, \eqref{eq:pcal-onevarKpq} and \eqref{eq:pcal-twovarHpql}.}
    \label{tab:typesofb2withmodfun}
\end{table}
 We will discuss this via examples. In the left column of table \eqref{tab:typesofb2withmodfun}, the exponents of $r = \abs\xi$, $k$ and $b_0$ are subjected to some relations due to the homogeneity of $b_2$. Namely, $b_2$ is of degree $-4$ in $r$ while $b_0 = (k^2r^2 -\lambda)^{-1}$ is of degree $-2$, therefore we have 
\begin{align*}
    \nu_1 = \tilde \nu_2 = 2(p+q)-4, \,\,\,  \nu_0 = 2p-4
    \,\,\, \nu_2 =  2(p+q+l)-4.
\end{align*}
Also the subscript of $\nu$ and $\mu$ stands for the number of arguments of the modular curvature function, for instance, in the first row, the modular curvature function is of one-variable, thus we use $\nu_1$ and $\mu_1$. The exponent of $k$ depends on $\nu$:
\begin{align*}
   \tilde \mu_2  = \nu_2,\,\,\, \mu_j = \nu_j+2-j, \, j=0,1,2,
\end{align*}
so that one will obtain the correct power of $k$ appeared in the right column. A conceptual reason for those exponents is that once we rewrite the modular scalar curvature in terms of $\log k$, the power of $k$ in front of the modular functions will become the same, which is equal to $(2-m)$. \par

We shall use the second row in table \ref{tab:typesofb2withmodfun} as an example to show how to achieve the left column. Recall the notations: $r = \abs\xi$, $m = \dim M$, $j_0 = m/2 -1$, $b_0 = (k^2 r^2 - \lambda)^{-1}$ and we need to compute 
\begin{align*}
    \int_0^\infty (\frac{d}{d\lambda})^{j_0}\Big|_{\lambda=-1} b_2(r,\lambda) r^{m-1}dr.    
\end{align*}
For the explicit term in second row of the table, the integral above becomes with respect to the substitution $r^2 \mapsto r$:
\begin{align*}
    k^{\mu_2} \int_0^\infty\left (\frac{d}{d\lambda}\right)^{j_0}\Big|_{\lambda=-1}
    (k^2 r - \lambda)^{-p} (\nabla k) (k^2 r - \lambda)^{-q} (\nabla k)
    (k^2 r - \lambda)^{-l}  r^{\nu_2/2} r^{(m-2)/2}(dr/2). 
\end{align*}
To apply the rearrangement lemma, we match each $r$ with a $k^2$:
\begin{align*}
   k^\beta \int_0^\infty\left (\frac{d}{d\lambda}\right)^{j_0}\Big|_{\lambda=-1}
    b_0^p (\nabla k) b_0^q (\nabla k) b_0^l(k^2r)^{\nu_2/2} (k^2r)^{(m-2)/2}(dk^2r/2),
\end{align*}
where  $b_0 = (k^2 r - \lambda)^{-1}$ and 
\begin{align*}
    \beta = \mu_2 -\nu_2 + (2-m) -2 = -m.
\end{align*}
Therefore we have explained the appearance of the factor $k^{-m}$ in the second row of table \ref{tab:typesofb2withmodfun}. To see the modular curvature function $H_{(p,q,l,m)}$, the rearrangement lemma tells us that (with a substitution $x = k^2r$ in mind) the integral is equal to
\begin{align*}
    k^{-m} H(\modop_{(1)},\modop_{(2)}) (\nabla k \nabla k), 
\end{align*}
with the function $H(s,t)$:
\begin{align*}
    &H(s,t)\\
    =&\,\, \int_0^\infty\left (\frac{d}{d\lambda}\right)^{j_0}\Big|_{\lambda=-1}
    (x-\lambda)^{-p} (sx-\lambda)^{-q} (stx-\lambda)^{-q} x^{\nu_2/2} x^{(m-2)/2} (dx/2).
\end{align*}
To continue the computation, set $u=1/x$:
\begin{align*}
    H(s,t)&=
    \int_0^\infty 
    u^{-j_0} \left (\frac{d}{d\lambda}\right)^{j_0}\Big|_{\lambda=-1}
    (1-u\lambda)^{-p} (s-u\lambda)^{-q} (st-u\lambda)^{-l} (du/2) \\
    &= \int_0^\infty \left (\frac{d}{dz}\right)^{j_0}\Big|_{z =-u}
  (1-z)^{-p} (s-z)^{-q} (st-z)^{-l} (du/2) \\
  &=  \int_{-\infty}^0 \left (\frac{d}{du}\right)^{j_0}
  (1-u)^{-p} (s-u)^{-q}  (st-u)^{-l} (du/2) \\
  &= \frac12 \left (\frac{d}{du}\right)^{j_0-1} \Big|_{u=0}(1-u)^{-p} (s-u)^{-q}  (st-u)^{-l}\\
  &=\frac12 H_{(p,q,l,m)}(s,t).
\end{align*}
A few words about the calculation above: since $m\ge 4$, $j_0-1\ge 0$, also in our later computation, the integers $p,q,l$ are all positive, in particular the function in the last line vanishes at $u=-\infty$. The coefficient $1/2$ is ignored in table \ref{tab:typesofb2withmodfun} since it appears for all terms thus can be moved to the overall factor.  \par

The first and third row of table \ref{tab:typesofb2withmodfun} can be justified by in a similar way. The last row is special since it is reduced to the commutative case. First, we can repeat the argument above by splitting $q = \alpha+\beta$, since $S_\Delta$ commutes with $b_0$, we can rewrite $\abs\xi^{\nu_0} k^{\mu_0} b_0^p S_\Delta$ as $\abs\xi^{\nu_0} k^{\mu_0} b_0^\alpha S_\Delta b_0^\beta$. Repeat the calculation,  we obtain:   
\begin{align*}
    \abs\xi^{\nu_0} k^{\mu_0} b_0^\alpha S_\Delta b_0^\beta \mapsto
    k^{2-m} K_{(\alpha,\beta,m)}(\modop)(S_\Delta).
\end{align*}
Notice that the modular operator acts trivially (as identity) on $S_\Delta$, so $K_{(\alpha,\beta,m)}(\modop)(S_\Delta) = K_{(\alpha,\beta,m)}(1)S_\Delta$, while
\begin{align*}
    K_{(\alpha,\beta,m)}(1) & = \brac{\brac{\frac{d}{du}}^{m/2-2}(1-u)^{-\alpha} (s-u)^{-\beta}}\Big|_{u=0,s=1} \\
    &=  \brac{    \brac{\frac{d}{du}}^{m/2-2} (1-u)^{-\alpha-\beta} } \Big|_{u=0}\\
    &= \frac{1}{(p-1)!} \brac{\frac{d}{du}}^{p-1+m/2-2}\Big|_{u=0} (1-u)^{-1} \\
    &=  \frac{(p-2+(m-2)/2)!}{(p-1)!} .
\end{align*}
On the other hand, we recall the approach used in the commutative situation. Start with the contour integral, since 
\begin{align*}
    \frac{1}{2\pi i} \int_C e^{-\lambda} b_0^p d\lambda & = 
    \frac{1}{2\pi i} \int_C e^{-\lambda} \frac{1}{(p-1)!}\frac{d^{p-1}}{d\lambda^{p-1}} b_0 d\lambda\\
    &= \frac{1}{(p-1)!} \frac{1}{2\pi i} \int_C e^{-\lambda} b_0 d\lambda  = \frac{1}{(p-1)!} e^{-k^2r^2}.\\
\end{align*}
We continue
\begin{align*}
   &\,\, k^{\mu_0}  \int_0^\infty r^{\nu_0} e^{-k^2 r^2} r^{m-1}dr \\
   =&\,\, k^{\mu_0} (k^2)^{-\nu_0/2}k^{2-m} k^{-2}\int_0^\infty (k^2 r)^{\nu_0/2} e^{-k^2 r} (k^2r)^{(m-2)/2} (dk^2r/2) \\
   =&\,\, k^{2-m}\frac12 \Gamma(\nu_0/2 + (m-2)/2 + 1).
\end{align*}
Here we have used  the relation $\mu_0 =\nu_0+2$. Note that $\nu_0 = 2p-4$, $\Gamma(\nu_0/2 + (m-2)/2 + 1) = (p-2+(m-2)/2)!$, therefore by ignoring the overall constant $1/2$, we have confirmed again, $c_{p,m}$  equals  $(p-2+(m-2)/2)!/(p-1)!$.

We now furthermore show that $c^{(m)}_{\Delta_k}$ the coefficient of the scalar curvature term  agrees with \eqref{eq:introv2-V2term-local}.
Indeed, we first compute the constant shown in Figure \ref{fig:scaconst}: 
\begin{align}
   \label{eq:modcur-constcm} 
   c^{(m)}_{\Delta_k} = -\frac14 c_{(2,m)} + \frac{2 c_{(3,m)}}{3m}
    =( -\frac14+ \frac{1}{2\cdot 3}) (m/2-1)! = -\frac{1}{12}\Gamma(m/2).
\end{align}

On the other hand, we have an overall factor:
\begin{align*}
    \frac{1}{(2\pi)^m}\cdot \frac12 \cdot \mathrm{Vol}(\mathbb S^{m-1}),
    \,\, \text{with}\,\,\,
    \mathrm{Vol}(\mathbb S^{m-1}) = \frac{2 \pi^{m/2}}{\Gamma(m/2)}.
\end{align*}
Finally,  the coefficient for the scalar curvature term is equal to 
\begin{align}
    \label{eq:modcur-coeffforscalcur}
   \frac{1}{(2\pi)^m} \frac{\mathrm{Vol}(\mathbb S^{m-1})}{2} \brac{
    -\frac14 c_{(2,m)} + \frac{2 c_{(3,m)}}{3m}
    }
    =\frac{1}{(4\pi)^{m/2}}(-\frac{1}{12}).
\end{align}
 Let us summarize the discussion above as a proposition.
\begin{prop}
    Keep the notations. Let $K^{(\dim M)}_{\Delta_k}(s)$ and $H^{(\dim M)}_{\Delta_k}(s,t)$ be 
    the local curvature functions appeared in \eqref{eq:scalcur-R-Deltak}. The function  $K^{(\dim M)}_{\Delta_k}(s)$ is given in \textup{Fig.} \eqref{fig:onevarfun} and $H^{(\dim M)}_{\Delta_k}(s,t)$ has two parts, listed in \textup{Fig.} \eqref{fig:twovarfunp1} and \textup{Fig.} \eqref{fig:twovarfunp2}.
\end{prop}

To end this section, we  write down explicitly the  modular curvature functions $K^{(\dim M)}_{\Delta_k}(s)$  in dimension $4,6$ and $8$.
\begin{align*}
    K^{(4)}_{\Delta_k}(s) &= -1/(2s)  \\
     K^{(6)}_{\Delta_k}(s) &= -\frac{2 \mathit{s}+\sqrt{\mathit{s}}+2}{3 \mathit{s}^2} \\
K^{(8)}_{\Delta_k}(s) &= -\frac{2 \mathit{s}^{3/2}+3 \mathit{s}^2+4 \mathit{s}+2 \sqrt{\mathit{s}}+3}{2 \mathit{s}^3}
\end{align*}
While for the two variable function, $H^{(\dim M)}_{\Delta_k}(s,t)$
\begin{align*}
   H^{4}_{\Delta_k}(s,t) &= \frac{1}{\mathit{s}^{3/2} \mathit{t}}, \\
   H^{6}_{\Delta_k}(s,t) &= 
   \frac{\left(4 \mathit{s}+\sqrt{\mathit{s}}+3\right) \mathit{t}+\left(2 \sqrt{\mathit{s}}+1\right) \sqrt{\mathit{t}}+4}{3 \mathit{s}^{5/2} \mathit{t}^2},
\end{align*}
and
\begin{align*}
   &\,\, H^{8}_{\Delta_k}(s,t) \\
  =& \,\, {\scriptstyle
  \frac{\left(4 \mathit{s}^{3/2}+3 \mathit{s}+2 \sqrt{\mathit{s}}+1\right) \mathit{t}^{3/2}+\left(2 \mathit{s}^{3/2}+6 \mathit{s}^2+5 \mathit{s}+\sqrt{\mathit{s}}+4\right) \mathit{t}^2+\left(8 \mathit{s}+3 \sqrt{\mathit{s}}+5\right) \mathit{t}+\left(4 \sqrt{\mathit{s}}+2\right) \sqrt{\mathit{t}}+6}{2 \mathit{s}^{7/2} \mathit{t}^3}.
  }
\end{align*}

\appendix 
\section{Covariant derivatives of $e^h$ after deformation}
The main goal of this section is to show that the results in \cite[sec. 6.1]{MR3194491} for noncommutative two tori survives in our deformed tensor calculus setting. To be precise, let $h = \log k$ or $k = e^h$ in the deformed algebra $C^\infty(M_\theta)$ with $h = h^*$. Recall the modular operator $\modop(x) = k^{-2} x k^2$ and its logarithm $\logmodop = -2[h,\cdot]$.  We would like to prove the following identities:
\begin{align}
    \label{eq:pcal-nablak2nablah}
    \nabla k &= k f_1(\logmodop)(\nabla h), \\
    \Delta k &= k \left[f_1(\logmodop)(\Delta h) + 2 g_2(\logmodop_{(1)}, \logmodop_{(2)})(\nabla h \nabla h) g^{-1}\right].
    \label{eq:pcal-Deltak2Deltah}
\end{align}
where the functions $f_1$ and $g_2$ are given in \eqref{eq:scalcur-f1g2} and \eqref{eq:scalcur-g2} respectively.  
Consider the family of automorphism $\alpha_t(x) = e^{it\slashed D} x e^{-it\slashed D}$, $t\in\R$, whose infinitesimal version is denoted by $\delta(x) = i[\slashed D, x]$, namely
\begin{align*}
    \delta(x) = \frac{d}{dt}\Big|_{t=0}\alpha_t(x). 
\end{align*}
From \cite[sec. 6.1]{MR3194491} or \cite[sec. 3]{leschdivideddifference} , we know that:
\begin{align}
    \label{eq:pcal-deltak2deltah}
    \delta k &= k f_1(\logmodop)(\delta h), \\
    \delta^2 k &= k \left[f_1(\logmodop)(\delta^2 h) + 2 g_2(\logmodop_{(1)}, \logmodop_{(2)})(\delta h \delta h)\right].
    \label{eq:pcal-delta^2k2delta^2h}
\end{align}
Since the one-form $\nabla k$ can be identified with an operator via Clifford multiplication $\mathbf c(\nabla k) = [\slashed D, k]$, equation \eqref{eq:pcal-deltak2deltah} implies \eqref{eq:pcal-nablak2nablah}. 
\par
To prove \eqref{eq:pcal-Deltak2Deltah}, we need the following fact (cf. \cite[Prop. 3.45]{berline1992heat}): that for any smooth function $f$, let $X_f = \grad f$, we have 
\begin{align*}
    [\slashed D, \mathbf c(df)] = [\slashed D,[\slashed D,f]] =-2\nabla^{\slashed S}_{X_f} + \Delta f,
\end{align*}
where $\nabla^{\slashed S}$ is the connection on the spinor bundle. Similar to \eqref{eq:pcal-nablak2nablah}, we have 
$\grad k = k F(\logmodop)(\grad h)$ and $\nabla^{\slashed S}_{X_k} = k F(\logmodop)(\nabla^{\slashed S}_{X_h})$.

\begin{align*}
    \Delta k & = [\slashed D,[\slashed D,k]] + 2 \nabla^{\slashed S}_{X_k}
    =\delta^2(k) + 2 \nabla^{\slashed S}_{X_k} \\
    &= k \left[F(\logmodop)(\delta^2 h) + 2 G(\logmodop_{(1)}, \logmodop_{(2)})(\delta h \delta h)\right]
    + 2 k F(\logmodop)(\nabla^{\slashed S}_{X_h}) \\
    &= k F(\logmodop) ( \delta^2 h+ 2 \nabla^{\slashed S}_{X_h}) + 2 kG(\logmodop_{(1)}, \logmodop_{(2)})(\delta h \delta h) \\
    &= k F(\logmodop)(\Delta h) + 2k G(\logmodop_{(1)}, \logmodop_{(2)}) (\nabla h \nabla h) g^{-1}.
\end{align*}
\par

\section{Integration over the cosphere bundle  $S^*M$} 
\label{app:intsphere}
\begin{lem}
    Let $m =\dim M$, $g^{-1}$ be the metric on $T^*M$, $S^*M$ is the associated cosphere bundle. In a local coordinate $(x,\xi)$,
    We identify the fiber at $x$ $S^*_x M \cong S^{m-1} \subset T^*_xM \cong \R^m$ with metric $(g^{-1})^{ij}|_x = \delta_{ij}$.
    \begin{align}
        \int_{S^*M} (D\abs\xi^2)^2 = \frac4m \abs\xi^2 \mathrm{vol}(S^{m-1}) g^{-1}, \,\,\,\,
  \int_{S^*M} (D^2\abs\xi^2) = 2  \mathrm{vol}(S^{m-1}) g^{-1}
        \label{eq:pcal-intDxi^2}
    \end{align}
    while
    \begin{align}
   \int_{S^*M}     (D\abs\xi^2)^2 (\nabla^2 \abs\xi^2) =0, \,\,\,\,
   \int_{S^*M} (D^2\abs\xi^2) (\nabla^2 \abs\xi^2) =\frac{\mathrm{vol}(S^{m-1})}{m} \frac43 S_\Delta \abs\xi^2,
        \label{eq:pcal-intDxinablaxi}
    \end{align}
    and
    \begin{align}
      \int_{S^*M} (D\abs\xi^2) (D^2\abs\xi^2)(\nabla^3 \ell) = -\frac{\mathrm{vol}(S^{m-1})}{m} \frac83 S_\Delta \abs\xi^2.
        \label{eq:pcal-Dxinablaxi-nablaell}
    \end{align}
\end{lem}
\begin{proof}
   \begin{align*}
        (D^2\abs\xi^2) (\nabla^2 \abs\xi^2) = - \frac43 \sum_{k,p,l}\xi_p\xi_l R_{pkkl}. 
    \end{align*}
    \begin{align*}
        (D\abs\xi^2) (D^2\abs\xi^2)(\nabla^3 \ell) = -\frac83 \sum_{k,p,l}\xi_p\xi_k R_{plkl}
    \end{align*}
\end{proof}

\begin{lem}
    Keep the notations in the previous lemma. 
    \begin{align}
     \label{eq:pcal-intsigmaD1}
        \int_{S^*M}(D\abs\xi^2)(D\sigma_{\slashed D})\sigma_{\slashed D} &= \mathrm{vol}(S^{m-1})\frac2m \abs\xi^2 g^{-1} + \kappa_1,
    \\
    \label{eq:pcal-intsigmaD2}
    \int_{S^*M}(D\sigma_{\slashed D})^2& =  \mathrm{vol}(S^{m-1}) g^{-1} +  \kappa_2,\\
    \int_{S^*M} \brac{(D\sigma_{\slashed D})\sigma_{\slashed D}}^2 
    & = \mathrm{vol}(S^{m-1})\frac{-(m-2)}{m} \abs\xi^2 g^{-1}+ \kappa_3
    \label{eq:pcal-intsigmaD3}
\end{align}
where $\kappa_j$, $j=1,2,3$ are  anti-symmetric contravariant two-tensors.
\end{lem}

\begin{proof}
    Choose a orthonormal basis, so that: 
    \begin{align*}
        g^{-1} = \sum_j e_j \otimes e_j, \,\,\,
 \sigma_{\slashed D}(\xi) =\sum_k i\xi_k \mathbf c(e_k). 
    \end{align*}
    Compute
    \begin{align*}
        (D\abs\xi^2)(D\sigma_{\slashed D})\sigma_{\slashed D} 
        = \sum_{j,k,l} ( 2\xi_j e_j)(i \mathbf c(e_k) e_k) (i\xi_l \mathbf c(e_l)).
    \end{align*}
    Recall that $\int_{S^{m-1}}\xi_j\xi_l =0$ for $j\neq l$, while for $j=l$,  $\int_{S^{m-1}}\xi_j^2 = \abs\xi^2 \mathrm{vol}(S^{m-1})/m$, also $\mathbf c(e_j)^2 = -1$:
    \begin{align*}
     &   \int_{S^*M}\brac{
       \sum_{j,k,l} ( 2\xi_j e_j)(i \mathbf c(e_k)e_k) (i\xi_l \mathbf c(e_l)) 
        } \\ =& \,\, \frac{\mathrm{vol}(S^{m-1})}{m} \brac{
        \sum_j 2 \abs\xi^2  e_j \otimes e_j - 
        2 \abs\xi^2 \sum_{j\neq k}  \mathbf c(e_j) \mathbf c(e_k) e_j \otimes e_k 
    },
    \end{align*}
    the first term gives rise to $\mathrm{vol}(S^{m-1})\frac2m \abs\xi^2 g^{-1}$ and the second term is a anti-symmetric tensor. \par

    Now we consider $(D\sigma_{\slashed D})^2$, since $(D\sigma_{\slashed D})_j = i\mathbf c(e_j)$
\begin{align*}
    (D\sigma_{\slashed D})^2 &=\sum_{j,k} i^2 \mathbf c(e_j) e_j \mathbf c(e_k) e_k \\
    &= \sum_j e_j \otimes e_j + \sum_{j\neq k} \mathbf c(e_j) \mathbf c(e_k) e_j \otimes e_k.
\end{align*}
After applying $\int_{S^*M}$, the first term becomes $\mathrm{vol}(S^{m-1}) g^{-1}$ in \eqref{eq:pcal-intsigmaD2} and the second term results an anti-symmetric tensor.

 To show the last one:  for $\brac{(D\sigma_{\slashed D}) \sigma_{\slashed D}}^2$,
    \begin{align*}
        \brac{(D\sigma_{\slashed D}) \sigma_{\slashed D}}^2_{jk}
        =  \sum_{p,l} \xi_p \xi_l \mathbf c(e_j) \mathbf c(e_p) \mathbf c(e_k) \mathbf c(e_l), 
    \end{align*}
  after integration, only those terms with $p=l$ will survive,  the result is given by:
  \begin{align*}
      \mathrm{vol}(S^{m-1}) \frac1m \abs\xi^2 \sum_{p}  \mathbf c(e_j) \mathbf c(e_p) \mathbf c(e_k) \mathbf c(e_p).
  \end{align*}

  For $j=k$, which corresponds to the symmetric part, 
  \begin{align*}
      \sum_{p=1}^m  \mathbf c(e_j) \mathbf c(e_p) \mathbf c(e_j) \mathbf c(e_p) =- (m-2) 
  \end{align*}
  For $j\neq k$:
  \begin{align*}
   \sum_p   \mathbf c(e_j) \mathbf c(e_p) \mathbf c(e_k) \mathbf c(e_p) = (m-2) \mathbf c(e_j)\mathbf c(e_k)
  \end{align*}
  which is anti-symmetric in $j,k$. Sum up,
  \begin{align*}
     \int_{S^*M} \brac{(D\sigma_{\slashed D})\sigma_{\slashed D}}^2 
     = \mathrm{vol}(S^{m-1})\frac{-(m-2)}{m} \abs\xi^2 g^{-1} + \text{an anti-symmetric tensor}.
  \end{align*}
   \end{proof}

   \section{List of complete terms of $b_2$}    
   \label{app:completeb2terrm}
   As being mentioned before, the explicit expression of  $b_2$ contains approximately $30$ terms. We use \textsf{Mathematica} to assist the computation and record the results and mains steps below from Figure \ref{fig:a_2(b_0,p_2)_partI} to Figure \ref{fig:a_0(b_0,p_0)}. The terms are grouped  according to the summands of $b_2$:
   \begin{align*}
        b_2&=\left[\mathit{a}_0\left(b_0,p_0\right)+\mathit{a}_0\left(b_1,p_1\right)+\mathit{a}_1\left(b_0,p_1\right) \right. \\
    &+ \left. \mathit{a}_1\left(b_1,p_2\right)+\mathit{a}_2\left(b_0,p_2\right)\right] (-b_0).
   \end{align*}

   The notations appeared in the figures are  consistent with notations used in the paper: 
   \begin{itemize}
       \item  $\mathtt{ B_0}$  stands for $b_0 = (k^2\abs\xi^2 - \lambda)^{-1}$ in the paper and $\sigma_{\mathrm D}$ is the symbol of the Dirac operator $\slashed D$;
       \item All the $\xi$ in the figures should be understood as $\abs\xi$;
       \item The horizontal and vertical differential: $\nabla k$ and $D\abs\xi^2$ are denoted as $\mathtt{\triangledown[k]}$ and $\mathtt{\mathcal D[\xi^2]}$;
       \item 
           The third column consists of modular curvature functions defined in equations \eqref{eq:pcal-contcp} to \eqref{eq:pcal-twovarHpql}. Namely, $\mathtt{Kmod[p,q, s,dimM]}$, $\mathtt{H[p,q,r, s,t,dimM]}$ and $\mathtt{c[p,dimM]}$ are associated to $K_{(p,q,m)}(s)$ and $H_{(p,q,r,m)}(s,t)$, and   $c_{(p,m)}$ respectively (cf. Eq. \eqref{eq:pcal-contcp}, \eqref{eq:pcal-onevarKpq} and \eqref{eq:pcal-twovarHpql}). 
       \item $\mathcal S$ (appeared in Figure \ref{fig:a_1(b_1,p_2)_partI} and \ref{fig:a_0(b_0,p_0)}, middle column) is the scalar curvature function.
   \end{itemize}
   Each Figure consists of three columns. The left column contains the original terms before any integrations. The middle column is obtained by integrating  the left  one over the unit cosphere bundle. In terms of programing, this step is merely apply the substitution rules listed in Table \ref{tab:intsph}. The result is upto an overall factor $\mathrm{Vol}(\mathbb S^{m-1})$, $m=\dim M$. The noncommutative factors appeared here are the derivatives of the conformal factor: $\nabla^2 k$ or $\nabla k$, they give rise to one and two variable modular curvature functions
    respectively. According to the rearrangement lemma in section \ref{sec:rearrangementlema}, passing from the middle  to the right column, one just need to collect the exponents of the $\mathrm B_0$ factors, for example, for the term $- \abs\xi^2 k^3 b_0^2 (\nabla^2 k) b_0$ (5th-row in Figure \ref{fig:a_2(b_0,p_2)_partI}) gives rise to the function $-K_{(2,1,m)}(s)$. Sometimes, one needs to move $k$ in front of all $\nabla k$ (or $\nabla^2 k)$ before applying the rearrangement
    lemma, that is
    \begin{align}
        \label{eq:app-listb2-exchange-order}
        (\nabla k) k = k k^{-1} (\nabla k) k = k \modop^{1/2}(\nabla k).
    \end{align}
    This gives rise to a function $\sqrt s$ or $\sqrt t$ depending the position of $\nabla k$. 
    For instance, look at Figure \ref{fig:a_1(b_1,p_2)_partII}, second row, we have:
    \begin{align*}
     -\abs\xi^2 k b_0 (\nabla k) b_0 (\nabla k) k b_0  \longmapsto
     -\sqrt s \sqrt t H_{(1,1,1,m)}(s,t), \,\, m = \dim M.
    \end{align*}
   The function $\sqrt s \sqrt t$ reflects the fact that in order to we  move the $k$ appeared at the right end infront, we have to apply \eqref{eq:app-listb2-exchange-order} twice. \par
As a reminder, there is an overall constant $1/2$ in front of all functions.

   \begin{figure}[!htpb]
    \centering
    \includegraphics[scale=0.6]{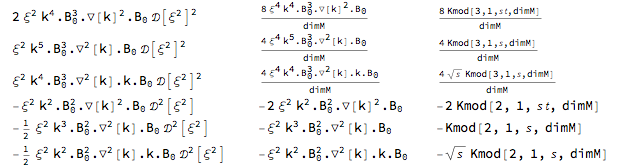}
    \caption{Terms come from $a_2(b_0,p_2)$, part I.}
    \label{fig:a_2(b_0,p_2)_partI}
\end{figure}
\begin{figure}[!htpb]
    \centering
    \includegraphics[scale=0.6]{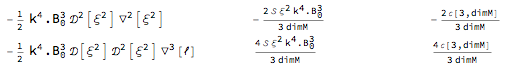}
    \caption{Terms come from $a_2(b_0,p_2)$, part II.}
    \label{fig:a_2(b_0,p_2)_partII}
\end{figure}
\begin{figure}[!htpb]
    \centering
    \includegraphics[scale=0.5]{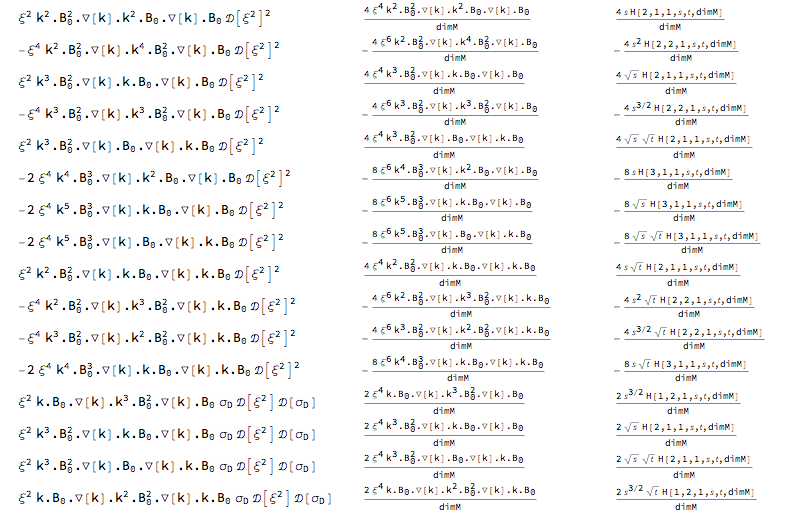}
    \caption{Terms come from $a_1(b_1,p_2)$, part I.}
    \label{fig:a_1(b_1,p_2)_partI}
\end{figure}

\begin{figure}[!htpb]
    \centering
    \includegraphics[scale=0.5]{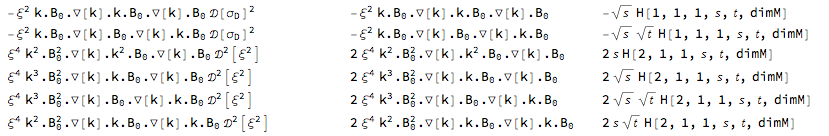}
    \caption{Terms come from $a_1(b_1,p_2)$, part II.}
    \label{fig:a_1(b_1,p_2)_partII}
\end{figure}

\begin{figure}[!htpb]
    \centering
    \includegraphics[scale=0.7]{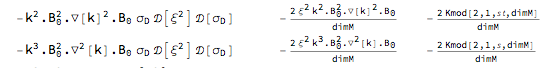}
    \caption{Terms come from $a_1(b_0,p_1)$}
    \label{fig:a_1(b_0,p_1)}
\end{figure}

\begin{figure}[!htpb]
    \centering
    \includegraphics[scale=0.57]{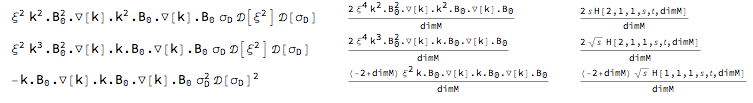}
    \caption{Terms from $a_0(b_1,p_1)$}
    \label{fig:a_0(b_1,p_1)}
\end{figure}

\begin{figure}[!htpb]
    \centering
    \includegraphics[scale=0.8]{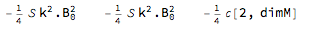}
    \caption{Terms from $a_0(b_0,p_0)$}
    \label{fig:a_0(b_0,p_0)}
\end{figure}

\begin{figure}[!htpb]
    \centering
    \includegraphics[scale=0.6]{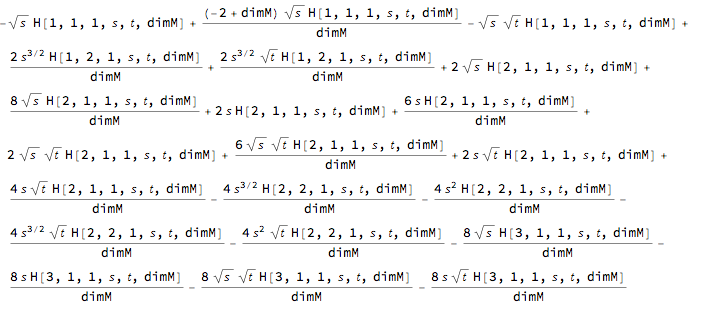}
    \caption{Two-variable local curvature function $H^{(m)}_{\Delta_k}(s,t)$, part I.}
    \label{fig:twovarfunp1}
\end{figure}

\begin{figure}[!htpb]
    \centering
    \includegraphics[scale=0.6]{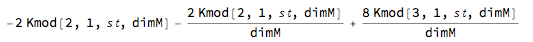}
    \caption{Two-variable modular curvature function $H^{(m)}_{\Delta_k}(s,t)$, part II.}
    \label{fig:twovarfunp2}
\end{figure}

\begin{figure}[!htpb]
    \centering
    \includegraphics[scale=0.47]{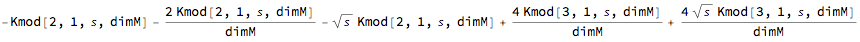}
    \caption{One-variable modular curvature function $K^{(m)}_{\Delta_k}(s)$.}
    \label{fig:onevarfun}
\end{figure}

\begin{figure}[!htpb]
    \centering
    \includegraphics[scale=0.6]{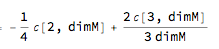}
    \caption{Coefficient for the scalar curvature term $c^{(m)}_{\Delta_k}$.}
    \label{fig:scaconst}
\end{figure}

\FloatBarrier

\section*{Acknowledgement}
The author would like to thank MPIM, Bonn for providing marvelous working environment. The author is also greatly indebted to Henri Moscovici for suggesting this  research project. This paper grew out of numerous conversations with him.

\bibliographystyle{hplain}

\bibliography{mylib}

\end{document}